\documentclass[12pt,a4paper]{amsart}

\usepackage{rotating, graphicx}
\usepackage{amscd}
\usepackage{pb-diagram}
\usepackage{url,crossreftools}
\usepackage[colorlinks=true, linkcolor=blue]{hyperref}
\usepackage[utf8]{inputenc}
\usepackage{amssymb}
\usepackage{colonequals}
\usepackage{wrapfig}
\usepackage{float}
\usepackage{appendix}
\usepackage{lscape}

\usepackage[usenames,dvipsnames]{xcolor}

\newcommand{\hadgesh}[1]{\textcolor{MidnightBlue}{\emph{#1}}}

\newcommand{\GG}{\mathcal{G}}
\newcommand{\HH}{\mathcal{H}}

\newcommand{\R}{{\mathbb R}}

\usepackage{tikz}
\usetikzlibrary{calc,arrows.meta}

\usepackage{amsthm}
\usepackage{amssymb,amsmath}
\newtheorem{Thm}{Theorem}[section]

\newtheorem{Defi}[Thm]{Definition}
\newtheorem{Lem}[Thm]{Lemma}
\newtheorem{Term}[Thm]{Terminology}

\newcommand{\outg}{\mathrm{out}}
\newcommand{\ing}{\mathrm{in}}
\newcommand{\outdeg}{\mathrm{oud}}
\newcommand{\indeg}{\mathrm{ind}}

\newcommand{\F}{{\mathbb F}}

\newcommand{\defeq}{\overset{\textup{def}}{=}}
\newcommand{\PP}{\mathcal{P}}
\newcommand{\LL}{\mathcal{L}}

\DeclareMathAlphabet\EuR{U}{eur}{m}{n}
\SetMathAlphabet\EuR{bold}{U}{eur}{b}{n}

\newcommand{\curs}{\EuR}

\newcommand{\Tr}{\mathrm{Tr}}

\newcommand{\chg}[1]{\textcolor{black}{#1}}

\usepackage[aboveskip=1pt,labelfont=bf,labelsep=period,justification=raggedright,singlelinecheck=off]{caption}

\usepackage{fancyhdr}
\author{Pedro Concei\c{c}\~ao}
\address{Institute of Mathematics, University of Aberdeen }
\email{p.rodriguesdaconceicao.19@abdn.ac.uk}
\author{Dejan Govc}
\address{Faculty of Mathematics and Physics, University of Ljubljana}
\email{dejan.govc@fmf.uni-lj.si}
\author{J\=anis Lazovskis}
\address{Riga Business School, Riga Technical University }
\email{janis.lazovskis@rbs.lv}
\author{Ran Levi}
\address{Institute of Mathematics, University of Aberdeen}
\email{r.levi@abdn.ac.uk}
\author{Henri Riihim\"aki}
\address{Institute of Mathematics, University of Aberdeen }
\email{henri.riihimaki@abdn.ac.uk}
\author{Jason P. Smith}
\address{Department of Mathematics and Physics, Nottingham Trent University}
\email{jason.smith@ntu.ac.uk}


\title[Neighbourhood in Dynamic Digraphs]{An application of neighbourhoods in digraphs to the classification of binary dynamics}
\begin{document}
\maketitle

\begin{abstract}
A binary state on a graph means an assignment of binary values to its vertices.  
A time dependent sequence of binary states is referred to as binary dynamics. 
We describe a method for the classification of binary dynamics of digraphs, using particular choices of closed neighbourhoods. Our motivation and application comes from neuroscience, where a directed graph is an abstraction of neurons and their connections, and where the simplification of large amounts of data is key to any computation. We present a topological/graph theoretic method for extracting information out of binary dynamics on a graph, based on a selection of a relatively small number of vertices and their neighbourhoods. We consider existing and introduce new real-valued functions on closed neighbourhoods, comparing them by their ability to accurately classify different binary dynamics. We describe a classification algorithm that uses two parameters and sets up a machine learning pipeline. We   demonstrate the effectiveness of the method on simulated activity on a digital reconstruction of  cortical tissue of a rat, and on a non-biological random graph with similar density.
\end{abstract}

\section{Introduction}
A \hadgesh{binary state} on a graph means an assignment of binary values to its vertices. A motivating example  in this article appears in the context of neuroscience. If one encodes \chg{the connectivity of} a neuronal network  as a directed graph, then the spikes produced by the neurons at an instant of time is a binary state on the encoding graph.  Allowing time to vary and recording the spiking patterns of the neurons in the network produces an example of  a \hadgesh{binary dynamics} on the encoding graph, namely  a one-parameter family of binary states on its vertices. A network of neurons that receives external signals and responds to those signals thus generates a binary dynamics. Binary dynamics appear in other contexts as well \cite{Gleeson, Samuelsson-Socolar}, but in this paper we use networks of spiking neurons as a primary example.

The \hadgesh{signal classification problem}, i.e., the task of correctly pairing  a  signal  injected into a neuronal network with the response of the network, or in other words, identifying the incoming signal from the response, is generally very challenging. This paper proposes a methodology by which this task can be approached and provides scenarios in which this methodology is successful.

Considering raw binary states on a large graph  is generally quite problematic for a number of reasons. First, the sheer number of theoretically possible states makes analysing a collection of them a daunting task \cite{comm-concep-technical-advances,history-simu-neuro}. Moreover, natural systems such as neuronal networks tend to be very noisy, in the sense that the emerging dynamics from the same stimulus may take a rather large variety of forms \cite{dim-reduc-neural, Stein-Gossen-Jones}. Finally, it is a general working hypothesis in studying network dynamics that the network structure affects its function \cite{connect-brain-func,higher-order-syn,connect-dynamics-curto-morrison,Rubinov-Sporns}. \chg{This paradigm in neuroscience is often encapsulated by the slogan \emph{``neurons that fire together tend to wire together''}. Hence, when studying dynamics on a neuronal network, it makes sense to examine assemblies of vertices, or subgraphs, and the way in which they behave as dynamical sub-units, instead of considering individual vertices in the network \cite{topo-cell-assembly,cell-groups,network-motifs}.} 

In previous studies we considered cliques in a directed graph, with various orientations of the connections between nodes, as basic units from which one could extract information about binary dynamics \cite{GLS,Fund}. However, the results in these papers fell short of suggesting an efficient classifier of binary dynamics \cite[Sections 4.1-4.2]{GLS}. Indeed, when we applied the methods of \cite{GLS,Fund} to the main dataset used in this paper, we obtained   unsatisfactory classification accuracy. This suggests  that in a graph that models a natural system  cliques may be too small to carry the amount of information required for classification of a noisy signal. \chg{This motivates us to build our classification strategy on neuron assemblies, where the richer structure serves a dual purpose of amalgamating dynamical information and regulating the noise inherent in single neurons or cliques.}

\chg{The guiding hypothesis of this paper is that a collection of vertex assemblies, forming a subgraph of the ambient connectivity graph encoding a network, can be used in classification of binary dynamics on the network. A network of spiking neurons is our primary example. Taking this hypothesis as a guideline, we introduce a very flexible  feature generation methodology that takes as input  binary dynamics on a digraph $\GG$  induced on a  preselected collection of subgraphs of $\GG$,  and turns it into a feature vector, which can then be used in machine learning classification.  The neighbourhood of a vertex $v$ in the graph $\GG$, namely the subgraph of $\GG$ that is induced by $v$ and all its  neighbours in $\GG$, suggests itself naturally as a type of subgraph to be considered in this procedure,  and is a central object of study in this paper. Vertex neighbourhoods have been studied extensively in graph theory and its applications \cite{Bianconi}. 
An outline is given below and a full description in Methods.}

The way we apply the method can be summarised as follows. Given a directed graph $\GG$ we use a variety of real valued vertex functions that we refer to as \chg{\hadgesh{selection parameters}} and are derived from the neighbourhood of each vertex,  to create a sorted list of the vertices. With respect to each such parameter, we pick the ``top performing'' vertices and select their neighbourhoods. To that collection of subgraphs we apply our feature generation method, \chg{which is based again on applying the same parameters to the selected neighbourhoods, now in the role of \emph{feature parameters}.  All the parameters we use are invariant under isomorphism of directed graphs, i.e. graph properties that remain unchanged when the vertices are permuted while leaving their connectivity intact. Therefore we occasionally refer to certain parameters as ``graph invariants".}

\chg{The choice of parameters is related to measures of  network connectivity and architecture. For instance, the parameters  {\bf fcc} and {\bf tcc} (see Table \ref{tab:parameters}) are examples of measures of functional segregation \cite{Rubinov-Sporns}.  The parameters we refer to as \emph{spectral parameters} arise in spectral graph theory \cite{Chung} and are prevalent in many applications, including in neuroscience. For instance, the paper \cite{Laplacian} studies the Laplacian spectrum of the macroscopic anatomical neural networks of  macaques and cats, and the microscopic network of the C-elegans. The topological parameters, such as the Euler characteristic  {\bf ec} and Betti numbers are classical topological invariants. In \cite{Fund}  these were used in various ways to extract information on structure and  function and their interaction in the Blue Brain Project reconstruction on the neocortical column. The parameter {\bf size} is a natural parameter associated to any graph and is closely related to firing rate in neuroscience. However, most of the parameters we tested were never examined in a neuroscientific context. Our aim was to investigate which parameters may prove useful in classification of binary dynamics without making any assumptions about their relevance. It is exactly this approach that allowed us to discover that certain spectral parameters perform strongly as selection parameters, while others do not. At the same time a newly introduced topological parameter, ``normalised Betti coefficient'' {\bf nbc} shows strong performance as a feature parameter when tested on neighbourhoods with low selection parameter values, but not on high selection values.}

The primary test of our methods in this paper is done on  data generated by the \href{https://www.epfl.ch/research/domains/bluebrain}{Blue Brain Project} that was also used in \cite{Reimann-abdn} for signal classification by established neuroscience methodology. The data consists of eight families of neuronal stimuli that are injected in a random sequence to the  digital reconstruction of the neocortical column of a young rat. This reconstructed microcircuit  consists of approximately 31,000 neurons and 8,000,000 synaptic connections, and is capable of receiving neuronal signals and responding to them in a biologically accurate manner \cite{Cell_paper}.  We used 60\% of the data to train a support vector machine, and the remaining 40\% for classification. With our methods we are able to achieve classification accuracy of up to 88\%. 

\chg{In this paper we did not attempt to explain the relevance of any of the mathematical concepts we use to neuroscience, as our main aim was to discover and investigate the utility of various concepts. However, in} \cite{Reimann-abdn}  the same dataset is studied by  standard techniques of computational neuroscience combined with the ideas presented in this paper. In particular, it is shown  that an informed choice of neighbourhood improves  classification accuracy  when compared to traditional methods. Interestingly, selection of neighbourhoods that improved performance with the technique presented in \cite{Reimann-abdn}  show reduced performance with the techniques presented in this article, and vice versa. In both projects a classification accuracy of nearly 90\% was achievable, but with  different selection parameters (see Results). This suggests that  considering vertex neighbourhoods as computational units can be beneficial in more than one way.

To further test our methods in different settings, we used the NEST - Neural Simulation Tool \cite{NEST} to generate neuronal networks. This software package simulates network models of spiking neurons using simplified neuron models to allow more flexibility and faster processing speed. We created a collection of eight families of stimuli, but on  random graphs with varying densities, and applied our machinery to that dataset. Here again we obtained classification accuracy of up to 81\%.

Important work on (open) vertex  neighbourhoods was reported recently in \cite{Bianconi}. Our approach is independent of this work and is different from it in a number of ways. Most significantly, we do not study the structure of the entire graph and its dynamical properties by means of its full neighbourhood structure. Instead, we  aim to infer dynamical properties of the graph from a relatively small collection of vertices, selected by certain graph theoretic and topological properties, and their neighbourhoods.

High resolution figures and supplementary material is available at the \href{https://homepages.abdn.ac.uk/neurotopology/neighbourhoods}{Aberdeen Neurotopology Group} webpage. In particular, we included a comprehensive visualization of spectral graph invariants of the Blue Brain Project graph, as well as other types of stochastically generated graphs, animations of some of the background work for this project, and a list of links to software implementing the methodology described in this paper.


\section{Results}
\label{sec:Results}

\chg{We start with a brief description of the mathematical formalism used in this article and our approach to  classification tasks. This is intended to make the section accessible to readers without a strong mathematical background. We then proceed by describing our main data source and the setup and implementation of our experiments.  Following this preparation we present our results, validation experiments, and an application of the same techniques in a different setup.}

\subsection{A brief introduction to the mathematical formalism}

\chg{In this article a \hadgesh{digraph} will always mean a finite collection of vertices (nodes) $V$ and a finite collection of oriented edges (arcs)  $E$. Reciprocal edges between a pair of vertices are allowed, but  multiple edges in the same orientation between a fixed pair of vertices and self-loops  are not allowed. }

\chg{The  fundamental mathematical concept essential for our discussion is that of the  neighbourhood of a vertex in a digraph; Figure~\ref{Fig:clsd-nbd}. Let $\GG$ be a digraph, and let $v_0$ be any vertex in $\GG$. The \hadgesh{neighbours} of $v_0$ in $\GG$ are all vertices that are ``one step away'' from $v_0$, in either direction. The  \hadgesh{neighbourhood} of $v_0$ in $\GG$ is the subgraph of $\GG$ induced by $v_0$ and all its neighbours, which we denote by $N_\GG(v_0)$. The vertex $v_0$ is referred to as the \emph{centre} of its neighbourhood.}

\begin{figure}
\begin{center}
\includegraphics[scale=1]{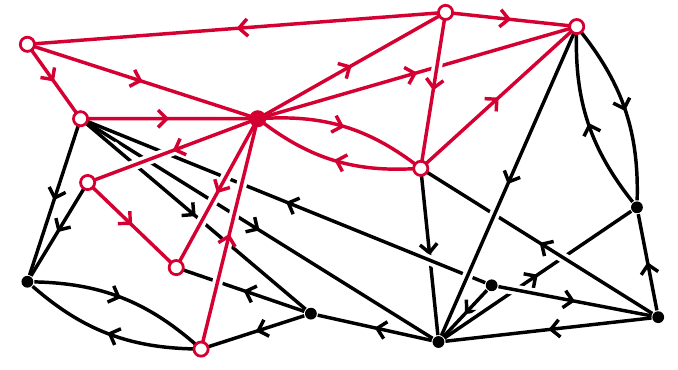}
\end{center}
\caption{A   neighbourhood in a digraph, marked in red, with its centre marked solid colour.}
	\label{Fig:clsd-nbd}
\end{figure}

\chg{Numerical invariants of digraphs can be found in pure and applied graph theory literature, many of those found their uses in theoretical neuroscience (see \cite{Rubinov-Sporns} for a good survey). Some such invariants are used in this article, and a few are introduced here for the first time (e.g. transitive clustering coefficient). Other parameters we used are defined by using topological constructions that arise from digraphs. Such constructions are typically invariant under digraph isomorphism. Standard tools of algebraic topology can then be used to extract numerical invariants of graphs in ways that take emerging higher dimensional structure into account.}

\chg{There are many ways in which one can associate a topological space with a digraph. In this article we use the \hadgesh{directed flag complex}. It is a topological space made out of gluing together \hadgesh{simplices} in different dimensions, starting at 0-simplices (points), 1-simplices (edges), 2-simplices (triangles), 3-simplices (tetrahedra) etc. The $n$-simplices in a directed flag complex associated to a digraph are its directed $(n+1)$-cliques, namely the ordered subsets of vertices $\{v_0, v_1,..., v_n\}$, such that there is an edge from $v_i$ to $v_j$ for all $i < j$. Figure \ref{Fig-clique-complex} shows the directed flag complex associated to a small digraph. The directed flag complex was introduced and used for topologically analysing structural and functional properties of the Blue Brain Project reconstruction of the neocortical columns of a rat \cite{Fund}. \chg{The interested reader may find a comprehensive survey of directed flag complexes and other topological concepts in the Materials and Methods section of that paper.} If $v_0$ is a vertex in $\GG$, we denote by $\Tr_\GG(v_0)$ the directed flag complex of $N_\GG(v_0)$.}

\begin{figure}[h!]
\begin{center}
	\includegraphics[scale=1]{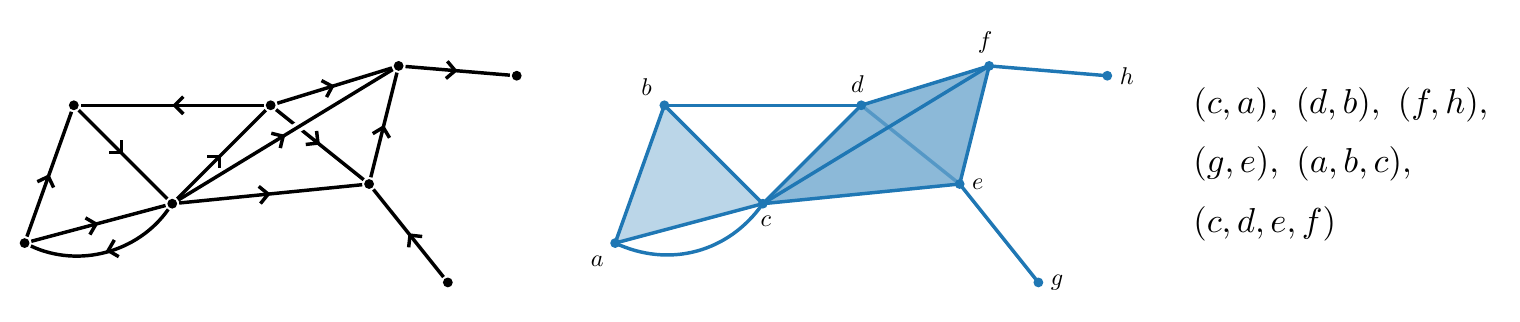}
\end{center}
	\caption{A digraph  (left), the associated directed flag complex  as a topological space (centre), and its maximal cliques  of  (right).}
	\label{Fig-clique-complex}
\end{figure}

\subsection{The classification method}

\chg{We now describe briefly our approach to classification of binary dynamics. For a precise mathematical definition of what we mean by binary dynamics see Methods. The task at hand can be described as follows. We are given a large set of instantiations of binary dynamics on a fixed digraph $\GG$, each of which is labelled by a symbol from some relatively small set. The label of each binary dynamic is unique and known. The aim is to produce a machine learning compatible topological summary for each binary dynamics, so that when the summaries are introduced in a random order, one can train on part of the data with known labels and predict the unknown labels of the remaining part. }

\begingroup
\begin{table}[h!]
\setlength{\tabcolsep}{10pt} 
\renewcommand{\arraystretch}{1} 
\begin{center}\begin{tabular}{l | l }
{\bf Abbreviation} & {\bf Short description}  \\
\hline
\hline
{\bf fcc} &  \parbox{10cm}{Clustering coefficient  (Fagiolo)} \\
\hline
 {\bf tcc} & Transitive clustering coefficient \\
\hline
 {\bf ec} & Euler characteristic  \\
\hline
{\bf nbc} & Normalised Betti coefficient   \\
 \hline
{\bf size} & \parbox{10cm}{Number of  vertices in the graph}  \\
\hline
{\bf asg} & Adjacency spectral gap   \\
\hline
{\bf asr} & Adjacency spectral radius  \\
\hline
{\bf blsg} & Bauer Laplacian spectral gap    \\
\hline
{\bf blsr} & Bauer Laplacian spectral radius   \\
\hline
{\bf clsg} & Chung Laplacian spectral gap    \\
\hline
{\bf clsr} & Chung Laplacian spectral radius    \\
\hline
{\bf tpsg}  & Transition probability spectral gap    \\
\hline
{\bf tpsr} & Transition probability spectral radius   \\
\end{tabular}
\end{center}
\caption{A partial list of the selection and feature parameters examined in this project. See Supplementary Material for additional parameters.}
\label{tab:parameters}
\end{table}
\endgroup

\chg{The \emph{first step} is selection of neighbourhoods. For each vertex $v$ in the digraph $\GG$ we consider its neighbourhood $N_\GG(v)$ and the associated directed flag complex $\Tr_\GG(v)$. We then compute a variety of numerical graph parameters of $N_\GG(v)$ and topological parameters of  $\Tr_\GG(v)$. These parameters are used to create a ranked list of vertices in $\GG$. We then select for each parameter 50 vertices that obtained the top (or bottom) values with respect to that parameter. We now have a set of 50 neighbourhoods corresponding to each parameter. A parameter that is used in this step is referred to as a \emph{selection parameter}, and we denote it by $P$. A short summary of the main parameters we used with their abbreviations  is in Table \ref{tab:parameters}. A detailed description of the parameters  is given in Methods. }

\chg{In the \emph{second step} we introduce binary dynamics in $\GG$. Each instantiation of the dynamics consists of several consecutive time bins (in our experiments we used two, but there is no limitation). For each time bin we consider the neurons that were active and the subgraph that they induce in each of the neighbourhoods we preselected. This gives us, for each selection parameter  and each time bin, a set of 50 subgraphs that correspond to a particular instantiation of binary dynamics on $\GG$. }

\chg{The \emph{third step} is vectorising the data, i.e., a computation of the same graph parameters and topological parameters for each of the subgraphs resulting from the second step. When we use our parameters  in the vectorisation  process  they are referred to as  \emph{feature parameters}, and are denoted  by $Q$. This now gives a vector corresponding to each instantiation of the dynamics, and the pair $(P,Q)$ of selection and feature parameters. }

\chg{The \emph{fourth and final step} is injecting the data into a support vector machine. In this project we used 60\% of the data for training and the remaining for testing. See Figure \ref{Fig:process} for a schematic summary of the process.}

\begin{figure}[h!]
\includegraphics[scale=.5]{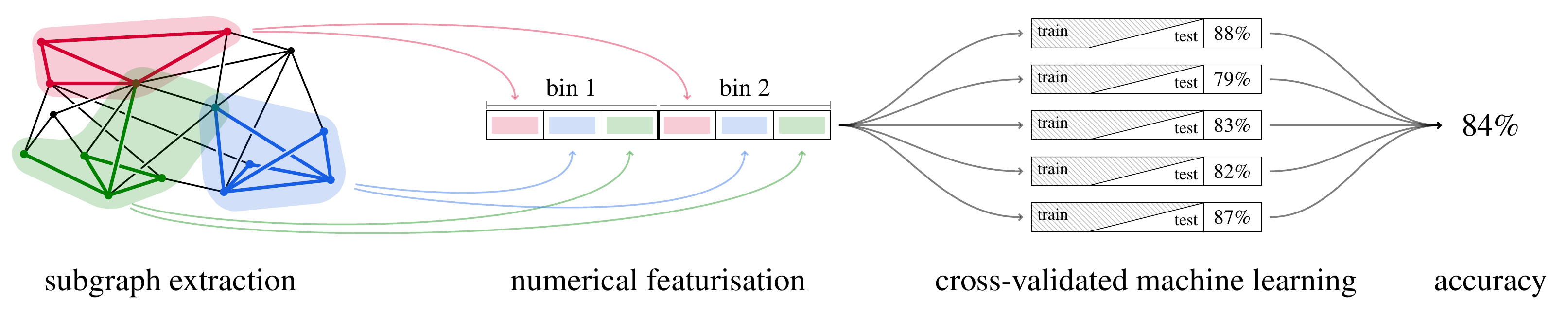}
\caption{A schematic description of the vector summary and classification pipeline.}
\label{Fig:process}
\end{figure}

\chg{We note that the method described here is an example of a much more general methodology that is described in detail in the Methods section of this article. In particular, the graph and topological parameters that we chose to work with are selected from within the abundance of mathematical concepts that arise in graph theory, combinatorics and topology. We do not attempt in this article to associate a neuro-scientific meaning to these parameters. }

\subsection{The data}

\chg{Our main source of data is a simulation that was run on a \href{https://www.epfl.ch/research/domains/bluebrain}{Blue Brain Project} reconstruction of the microcircuitry of the somatosensory cortex in the brain of a rat \cite{Cell_paper}. 
From this model we extract the connectivity of the microcircuit in the form of a digraph} whose vertices correspond to neurons, and with an edge from $v$ to $u$ if there is a synaptic connection from the neuron corresponding to $v$ to the one corresponding to $u$.  We denote the Blue Brain Project digraph by $\GG$. The digraph consists of  31,346 vertices and 7,803,528 edges. The connectivity matrix of this specific circuit, as well as 41 other instantiations of the reconstruction, is accessible on the \href{https://bbp.epfl.ch/nmc-portal/downloads}{Digital Reconstruction of Neocortical Microcircuitry}. 

\chg{The binary dynamics we experimented with consists of eight stimuli families labelled 0-7. For each stimulus a random subset (10\%) of afferent neurons is activated. The stimuli differ with respect to which subset of afferent neurons is activated, where afferents can be shared between stimuli. The probability of a given afferent being associated with two given stimuli is 1\%. In each stimulation time window one and only one stimulus is presented. The stimuli were injected into the circuit in a random sequence of 200 milliseconds per stimulus, and 557 repeats for each stimulus label. The dataset thus consists of  $4456$ binary dynamics functions. The task is to determine the label of that stimulus, i.e. the expected output is an integer from 0 to 7. Thus, the chance level performance would be 12.5\%. More detail on the source of data, biological analysis and an alternative approach to classification of the same data is in \cite{Reimann-abdn}. }

\subsection{Setup}\label{SS:setup}

We computed all the graph parameters \chg{listed in Table \ref{tab:parameters}, as well as additional parameters listed in Supplementary Material}, for all neighbourhoods in the digraph \chg{(see Supplementary Material - Data and Code, for a brief description of computational methods and links to software)}. We fixed a positive integer $M$, \chg{and for each selection parameter $P$ we selected the vertices $v_1, v_2, \ldots, v_M$, whose neighbourhoods $N_\GG(v_1),\ldots, N_\GG(v_M)$  obtained the top  (or bottom) $M$ values of the parameter $P$ (see Step \ref{step2} in Methods). We experimented with $M=20, 50, 100$ and 200.} Here we report on the results we obtained for $M=50$, which provided the highest classification accuracy. For $M=20$ performance was strong as well, but for  $M=100$ and 200 the improvement compared to $M = 50$ was relatively minor, and not worth the additional time and computation needed.

\subsection{Vector summaries}
\label{SS:binning-extraction}
Each binary dynamics  in our dataset has time parameter $t$ between $0$ and $200$ milliseconds. The  subinterval  $[0,60]$ is where almost all the \chg{spiking activity is concentrated across the interval.} Furthermore, the bulk of the stimulus is injected in the first 10ms. Since we aimed to classify the response to the stimulus rather than the stimulus itself,  we chose $\Delta = [10,60]$ and \chg{divided that interval into two 25ms subintervals}, as  experimentation showed that these choices provide the highest classification accuracy \chg{(see Step \ref{step1} in Methods)}.

\chg{We denote each instantiation of binary dynamics on $\GG$  by $B^n$, for $n=1, \ldots, 4456$. Each instantiation consists of two binary states $B^n_1, B^n_2$, corresponding to the neurons that fired in each of the 25ms subintervals. For each selection parameter $P$, and each of the corresponding neighbourhoods $N_\GG(v_m)$, $m=1, \ldots, 50$, we computed the subgraphs $N_{m,k}$ of $N_G(v_m)$ induced by the binary state $B^n_k$, that is, the subgraph induced by the neurons that fired in the given interval. This gave us, for each binary dynamics  $B^n$ and each graph parameter $P$,  a $2\times 50$  matrix $U_n^P$  of subgraphs of $\GG$, whose $(m,k)$ entry is $N^n_{m,k}$. (see Step \ref{step2} in Methods)}.

Finally, for each graph parameter $Q$ (from the same list of parameters) we applied $Q$ to the  \chg{entries of the matrix $U_n^P$} to obtain a numerical feature matrix $U_n^{P,Q}$ corresponding to the binary dynamics function $B^n$, the selection parameter $P$,  and the feature parameter $Q$. \chg{The matrix $U_n^{P,Q}$ is a vector summary of the binary dynamics $B^n$. (see Step \ref{step3} in Methods).}

\subsection{Classification}
For each pair of graph parameters $(P,Q)$ the  vector summaries $\{U_n^{P,Q}\}$  were  fed into a support vector machine (SVM) algorithm. \chg{Our classification pipeline was implemented in Python using the \texttt{scikit-learn} package and the SVC implementation therein. The SVC was initialised with default settings and we used a 60/40 train/test split. The kernel used was Radial Basis function. We used one-versus-one approach for multiclass classification.} For cross-validation we used standard 5-fold cross-validation in \texttt{scikit-learn}.
The results are presented in Figure \ref{Fig:classification}.

\begin{figure}[h!]
\includegraphics[scale=.45]{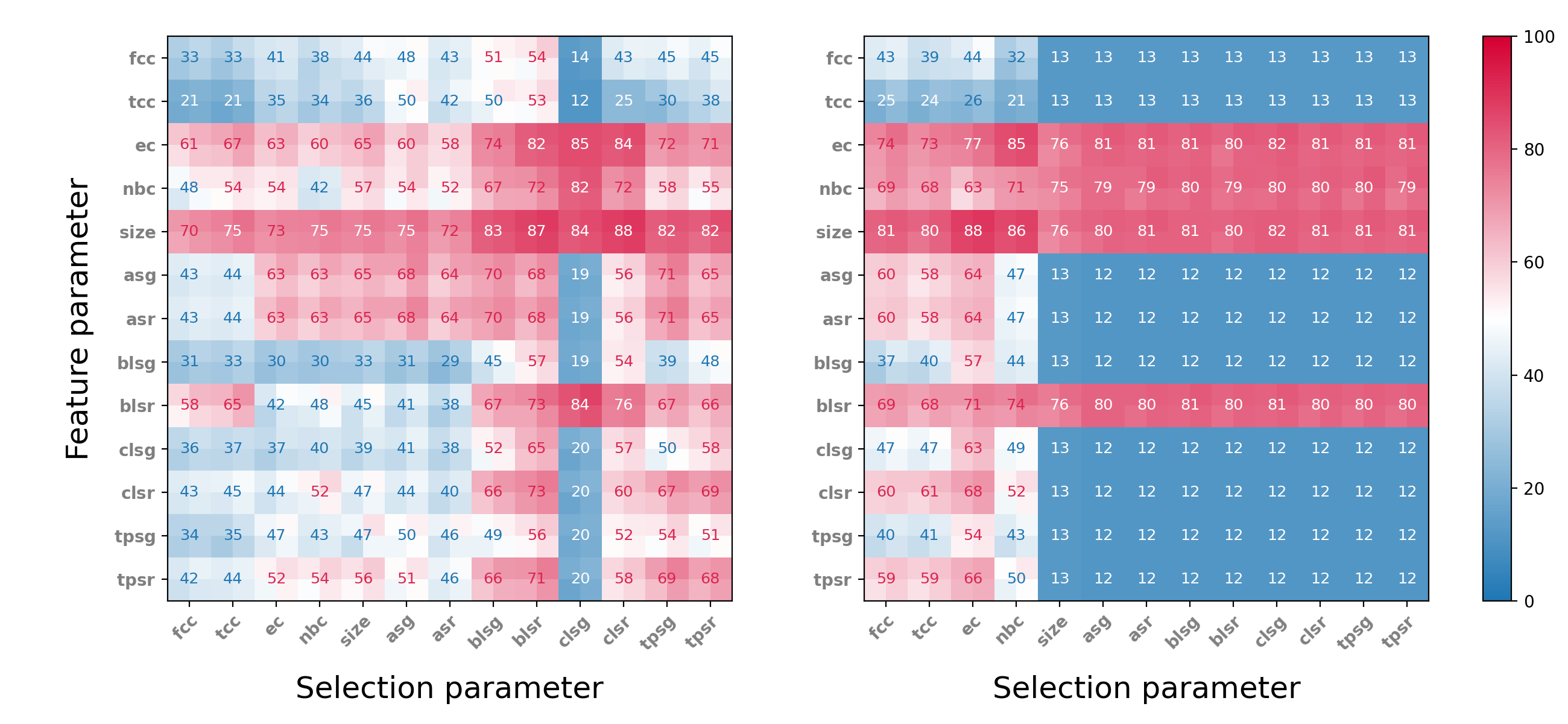}
\caption{\chg{Results of 8 stimuli classification experiments. Range of cross-validated accuracy is indicated by four smaller squares in each square. }Left: Classification accuracy selecting the 50 neighbourhoods with highest parameter value. Right: Classification accuracy selecting the 50 neighbourhoods with lowest parameter value. Compare with Supplementary Figure 3.}
\label{Fig:classification}
\end{figure}

\chg{For each of the selection parameters we tested, we considered both the  neighbourhoods that obtained the top 50 values and those that obtained the bottom 50 values.  In all the experiments, four parameters gave  markedly better performance when used as feature parameters than all other parameters: Euler characteristic ({\bf ec}), normalised Betti coefficient ({\bf nbc}), {\bf size} and  Bauer Laplacian spectral radius ({\bf blsr}).  All four perform significantly better than other feature parameters when the neighbourhoods were selected by bottom value parameters. With respect to top value selection parameters, {\bf ec} and {\bf size}, performed well, while {\bf nbc} and {\bf blsr} were significantly weaker as feature parameters, except when coupled with Chung Laplacian spectral gap ({\bf clsg}). }The neighbourhoods selected by top values of selection parameters gave best results when the selection parameter was one of the spectral graph invariants, while selecting by bottom value of selection parameters, the two types of clustering coefficients ({\bf cc} and {\bf tcc}) and Euler characteristic ({\bf ec}) performed best. 

Interestingly, the two best performing feature parameters, Euler characteristic and size, gave good results across all selection parameters, and performed almost equally well, regardless of whether the neighbourhoods were selected by top or bottom selection parameter value. This suggests that, at least  in this particular network, the choice of feature parameter plays a much more important role in classification accuracy than any specific selection parameter. On the other hand, examining the rows of the best performing feature parameters, in Figure  \ref{Fig:classification}, we see a difference of up to 27\% (top {\bf ec}), 40\% (top {\bf nbc}) and 18\% (top {\bf  size}) in classification accuracy, depending on which selection parameter is used, suggesting that, \chg{within a fixed choice of a feature parameter}, the selection parameter may play an important role in the capability of the respective neighbourhoods to encode binary dynamics. \chg{Note that randomly classifying the 8 stimuli gives an accuracy of 12.5\%.}

\subsection{Validation}
\label{sec:validation}
In order to validate our methods, we created  five experiments, the results of which we then compared to a subset of the original tests. In each case we retrained the SVM algorithm and then retested.

A motivating idea in neuroscience in general, and in this work in particular, is that structure is strongly related to function. Our approach, using neighbourhoods sorted by graph \chg{parameters} and  using the same graph \chg{parameters} as feature parameters is proposed in this article as a useful way of \chg{discovering combinations of parameters that achieve good} classification results of binary dynamics. To test the validity of this proposal, we challenged our assumptions in five different ways, as described below.

\subsubsection{Random selection.}
In this simple control experiment we test the significance of the selection parameter by comparing the results to a random choice of \chg{50 vertices and performing the same vector summary procedure on their neighbourhoods.} Twenty iterations of this experiment were performed, and  the results for each feature parameter were compared to the outcome for the same feature parameter and the  selection parameter with respect to which this feature parameter performed best. The results are described in Figure \ref{Fig:comparison-random}. 

\begin{figure}[h!]
\includegraphics[scale=.5]{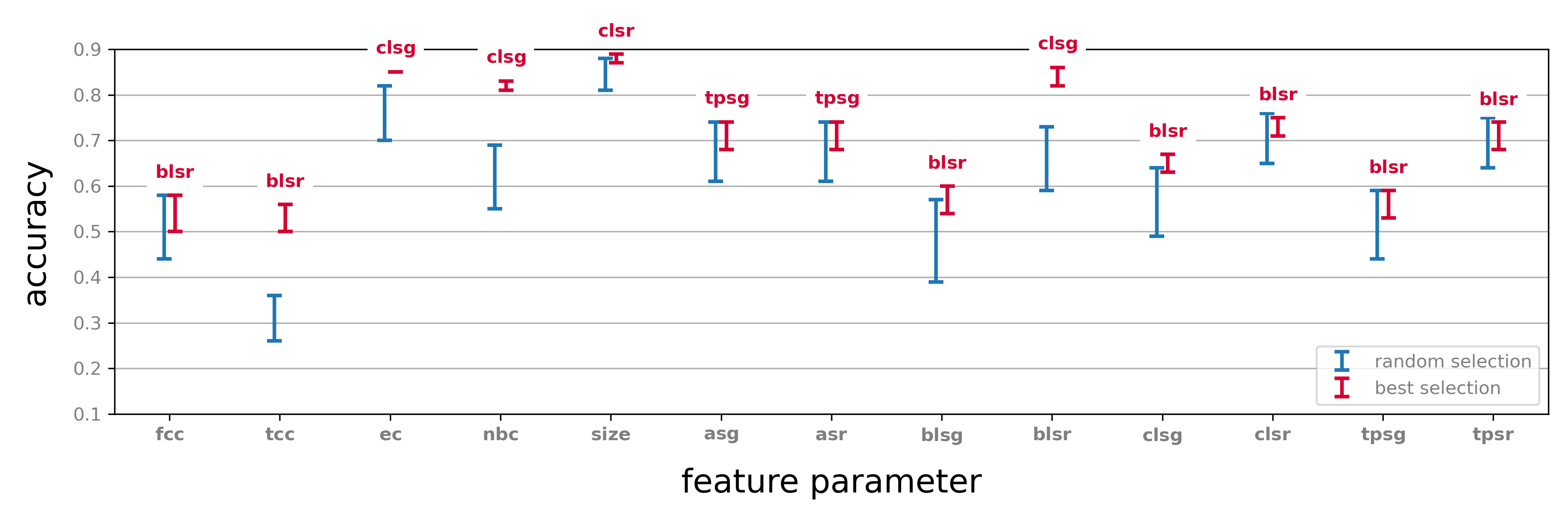}
\caption{\chg{The classification performance based on the neighbourhoods of 50 randomly selected vertices (blue), compared   to the performance of neighbourhoods selected by graph \chg{parameters} with respect to a selection of  feature parameters (red). Errors bars indicate range over 20 iterations. Labels on the red error bars indicate the selection parameter that performed best with respect to the indicated feature parameter. } Compare with Supplementary Figure 2.}
\label{Fig:comparison-random}
\end{figure}

We observe that in almost all cases reported here a choice of neighbourhoods determined by a selection parameter outperforms a random choice (in some cases marginally). We also note that in all those cases the performance of a choice informed by one of these selection parameters exhibits a more consistent behaviour in terms of classification accuracy. This can be seen from the considerably larger error bars in the case neighbourhoods are selected at random. On the other hand, for some feature parameters a random choice does not seem to be a disadvantage, even compared to the best selection parameter with respect to this feature parameter (Supplementary Figure 3). This suggests that while   selection and generation of vector summary by objective parameters are advantageous, experimentation is generally necessary in order to decide which parameters best fit the classification task.

\subsubsection{Neighbourhood vs. centre.}
A working hypothesis in this paper is that neighbourhoods carry more information about a binary dynamics than individual vertices. We examined for each selection of 50 neighbourhoods by a graph parameter, as described above, the classification capability of the centres of these neighbourhoods. \chg{Specifically, this experiment is identical to the original classification experiment, except for each selection parameter $P$ the two rows of the corresponding feature matrix  have binary values, where the $j$-th entry in row $i$ is set to be 1 if the $j$-th neuron in the sorted list fired in the $i$-th time bin at least once and 0 otherwise. }
These feature vectors were then used  in the classification task using the same train and test methodology. For each of the selection parameters we tested, we considered both the top 50 and the bottom 50 neurons in the corresponding sorted list. 

The results of this experiment were compared with the original experiments, and are shown in Figure \ref{fig:plot-val2}. We note that in all cases a very significant drop in performance occurs. Interestingly, some vertices in the top 50 of a sorted list show classification accuracy that is far better than random, while the bottom 50 give performance comparable to random (for example, {\bf fcc}). In some cases however, the bottom 50 vertices give better performance than the top 50. This suggests that the selection parameters play a role in classification accuracy even before considering the activity in the neighbourhood of a vertex. 

\chg{We also note that for almost all top valued selection parameters recorded in Figure  \ref{fig:plot-val2} and some of the bottom valued ones, the classification performance using the centre alone is significantly better than random. This observation reinforces the idea that selection parameters inform on the capability of neurons to inform on activity.}

\begin{figure}[h!]
	\includegraphics[width=\textwidth]{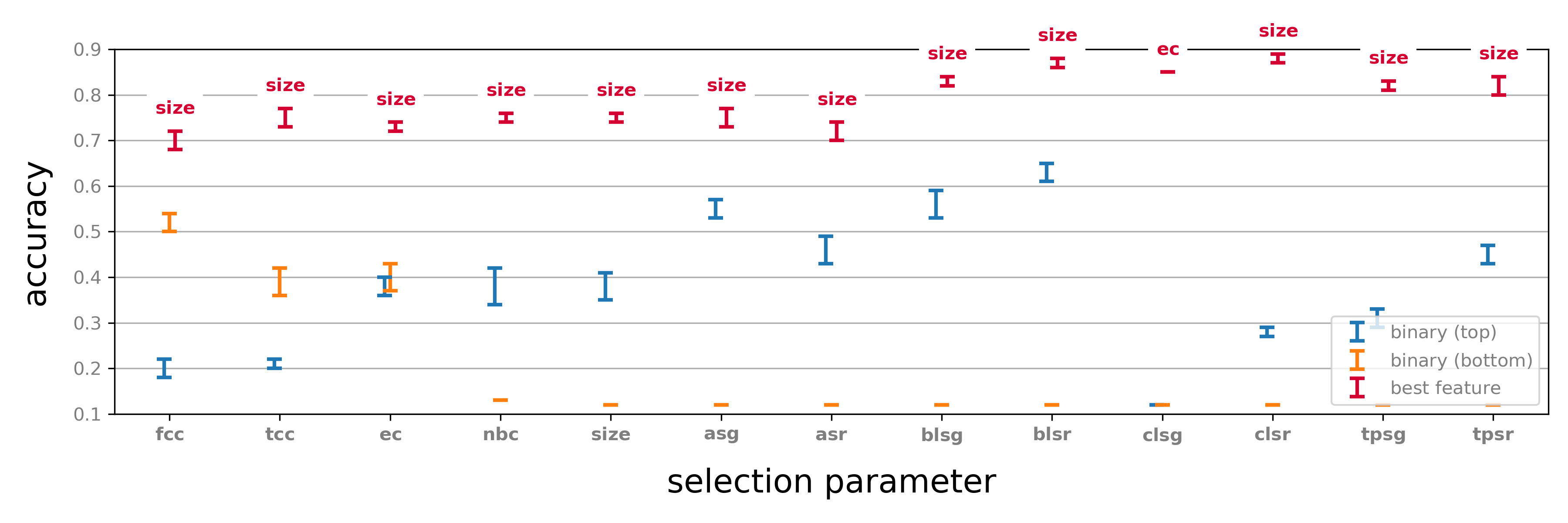}
	\caption{Classification results by binary  vectors using only the centres of each of the top and bottom 50 neighbourhoods for each parameter. For comparison, the performance for each selection parameter classified by the highest performing feature parameter is included.}
	\label{fig:plot-val2}
\end{figure}

\subsubsection{Neighbourhoods vs. arbitrary subgraphs.}
For each selection parameter we considered the degrees of the 50 selected centres. For a centre $v_i$ of degree $d_i$ we then selected at random $d_i$ vertices in the ambient graph and considered the subgraph induced by those vertices and the centre $v_i$. We used these 50 subgraphs in place of the original neighbourhoods. In this way we create  for each centre a new subgraph with the same vertex count as the original neighbourhoods that is unrelated to the centres in any other controllable way. We extracted feature vectors using these subgraphs for each of the selection parameters and repeated the classification experiment. The results were compared to the original results with respect to the strongest performing feature parameter. Notice that these are always either {\bf ec} or {\bf  size}, both of which can be applied to an arbitrary digraph, not necessarily a neighbourhood.

\begin{figure}[h!]
	\includegraphics[width=\textwidth]{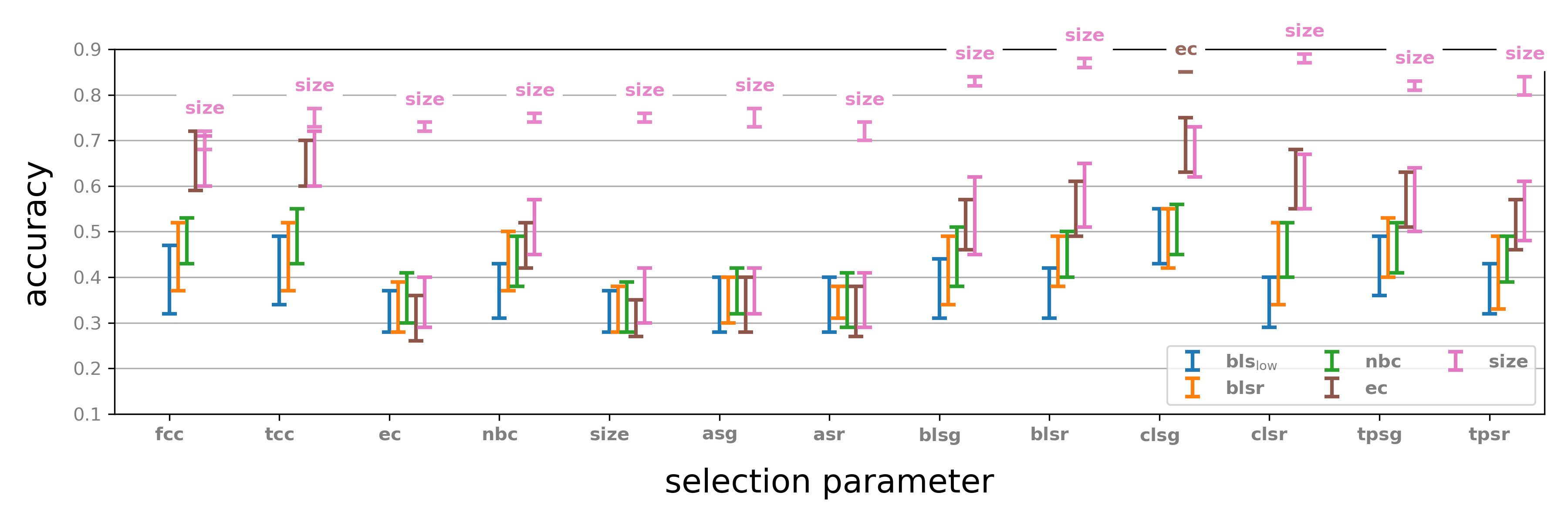}
	\caption{Classification by subgraphs of the same vertex count as the neighbourhoods selected by the specified selection parameters. The results of classification by the highest performing feature parameters are above each of the columns.}
	\label{fig:plot-val3.2}
\end{figure}

The results of this experiment were compared with the original experiments, and are shown in Figure \ref{fig:plot-val3.2}. There is a clear drop in performance for all selection parameters except {\bf fcc} (Fagiolo's clustering coefficient; See Methods). Furthermore, classification using these subgraphs shows considerably larger error bars. This suggests that using neighbourhoods with our methodology is advantageous. One explanation for this may be the tighter correlation of activity among neurons in a neighbourhood, compared to an arbitrary subgraph of the same size in the network, but we did not attempt to verify this hypothesis.

\subsubsection{Fake neighbourhoods.}
In this experiment we considered for each centre its degree and selected at random the corresponding number of vertices from the ambient graph. We then modified the adjacency matrix of the ambient graph so that the centre is connected to each of the vertices selected in the appropriate direction, so as to preserve the centre's  in- and out-degree. Computationally, this amounts to applying a random permutation to the row and the column of each of the centres. The result is a new ambient graph, where the old centres are now \chg{centres of}  new neighbourhoods. We extracted feature vectors using these ``fake neighbourhoods'' and repeated the classification experiment. The results were compared with the original classification. The outcome is illustrated in Figure \ref{fig:plot-val3.1}.

\begin{figure}[h!]
	\includegraphics[width=\textwidth]{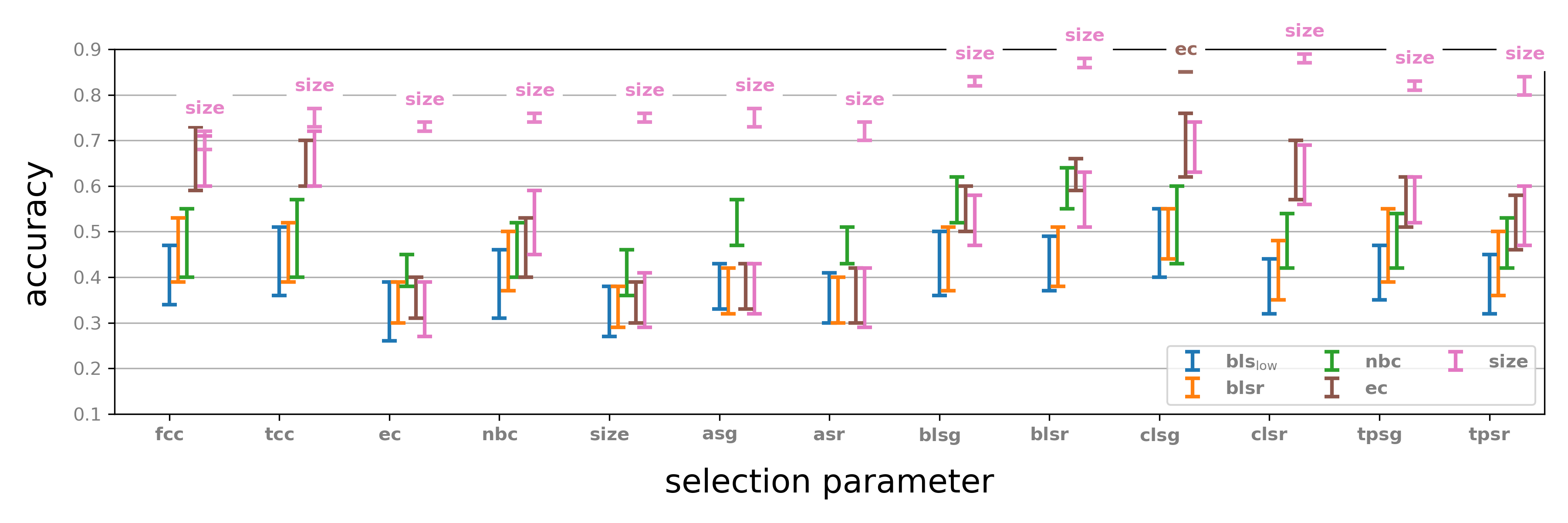}
	\caption{Classification by ``fake neighbourhoods'': Original classification with respect to best performing feature parameter is given for comparison.}
	\label{fig:plot-val3.1}
\end{figure}

We note that with respect to almost all selection parameters  there is a significant drop in performance resulting from this modification.  The one exception is {\bf fcc}, where {\bf ec} as a feature parameter actually sometimes gives slightly better results, but with a large error bar. It is interesting that the results are similar for some of the parameters to those observed in \chg{previous experiment (Figure \ref{fig:plot-val3.2})}, but quite different for others. However, the drop in performance is similar in both cases. We make no hypothesis attempting to explain these observations.

\subsubsection{Shuffled activity.}  In this experiment we applied a random permutation $\sigma$ of the neuron indices in the Blue Brain Project microcircuit, so that neuron  $\sigma(i)$ now receives the  spike train (sequence of spikes) of neuron $i$ for each stimulus. That is, we precompose the binary dynamics with $\sigma$ to get a new binary dynamics, which still appears in eight varieties, since the operation of permuting the neuron indices is bijective. In other words,  we can reconstruct the original activity from the shuffled activity by applying the inverse permutation $\sigma^{-1}$.  The same selection and feature parameters were used and the resulting data was used for training and testing. The  results are shown in Figure \ref{fig:plot-val4}.

\begin{figure}[h!]
	\includegraphics[width=\textwidth]{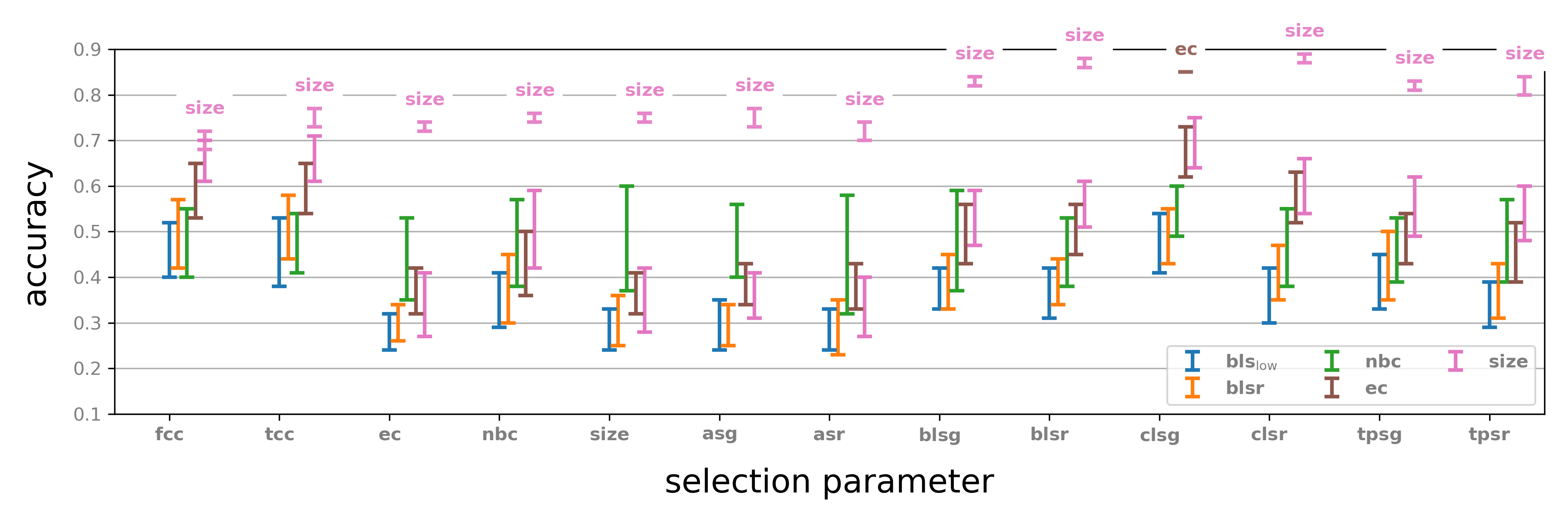}
	\caption{Classification of shuffled binary dynamics functions and comparison to the top results for the original dynamics. }
	\label{fig:plot-val4}
\end{figure}

We observe again that there is a significant drop in performance resulting from this shuffling. This is quite surprising since the shuffled activity spike train should give eight families of stimuli that carry some sort of internal resemblance, and since we retrained and tested with these stimuli, one could expect that the classification results will be comparable to those of the original experiments. That not being the case suggests that structure and function in the Blue Brain Project reconstruction are indeed tightly related.

\subsection{Testing the method on an artificial neuronal network}
To test our methods in a non-biological binary  dynamics setting, we conducted a set of experiments with the NEST simulator \cite{NEST}. The NEST software simulates spiking neuronal network models and offers a vast simplification of neuronal networks that are based on the exact morphology of neurons (such as the Blue Brain Project reconstructions). It also provides great flexibility in the sense that  it allows any connectivity graph to be implemented in it and any initial stimulation to be injected into the system with the response modulated by various flexible parameters. 
 
 \begin{figure}[h!]
\includegraphics[width=1\textwidth]{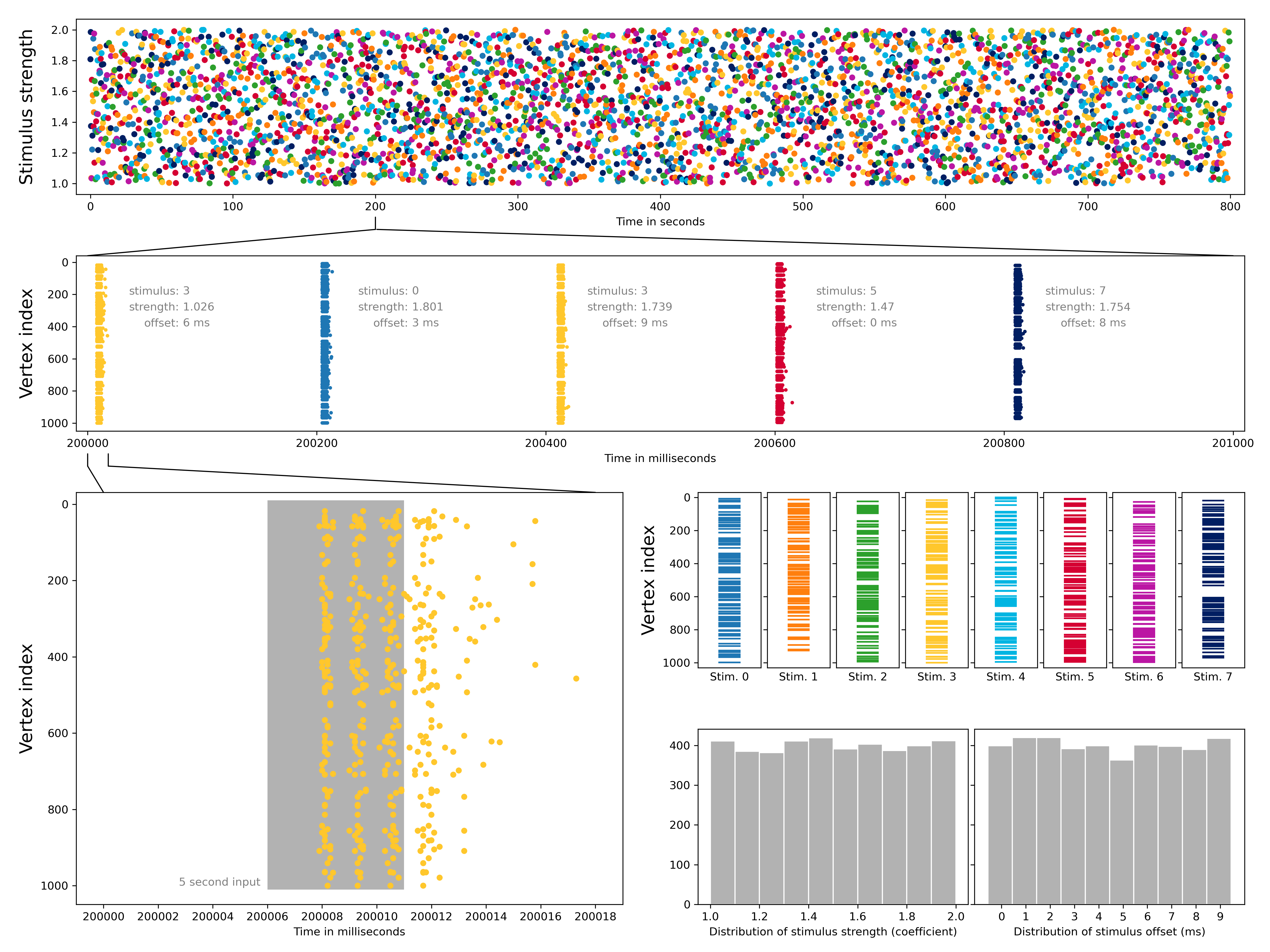}
\caption{\chg{Eight types of input stimuli for Erdős–Rényi random digraphs, executed as a single 800 second experiment. Top row: Sequence of stimuli types, 500 of each, and relative strength of input for each stimulus. Second row: Spiking neurons on a 1000 ms interval from the experiment. Bottom left: Spiking neurons and length of external input on a 18 ms interval. Third row right: Random selections of 100 vertices from 1000 vertices, acting as receptors of external input. Bottom row right: Distribution of randomly selected relative strength and input stimulus time offset over the whole experiment.}}
\label{fig:nest-visual}
\end{figure}

To move as far as possible from a strict biological setup, we generated a number of Erd\H{o}s--R\'enyi  random digraphs on 1000 vertices, which we implemented on NEST. We then created 8 distinct stimuli, each enervating a random selection of 100 vertices of the graph. A random sequence of stimuli was then created, with each stimulus type repeated 500 times. Our experiment consisted of injecting the sequence of stimuli into the simulator, for a duration of 5ms, one every 200 ms, \chg{to reduce the influence of one stimulus on the next.} To introduce some randomness, the start time of each stimulus is randomly selected from the first \chg{10ms} , the strength of each stimulus is multiplied by a random number between 1 and 2, and background noise is included (using NEST's \texttt{noise\_generator} device with strength 3). \chg{For each 200ms interval, the first 10ms were not included in the classification.  As a result some of the input may be included in the classified data, but never more than 4 ms, and for approximately 60\% of the 4000 stimuli the input is completely excluded from classification.} The code used to create these experiments is \href{https://github.com/jlazovskis/neurotop-nest/}{available online}, and the experiments are presented visually in Figure \ref{fig:nest-visual}. 

The spikes from this simulation were then extracted and were run through the same pipeline as the Blue Brain Project data. We experimented with graph densities of $0.08$, $0.01$ and $0.005$, and with selections of 10, 20, and 50 neighbourhoods. Figure \ref{fig:NEST} shows the performance by  the  selection parameters from Table \ref{tab:parameters}. Size was used in all cases as a feature parameter. The best performance was obtained with 50 neighbourhoods, with graph density of 0.01 in almost all selection parameters. The results of experiments with all parameters can be seen in Supplementary Figure 5.

\begin{figure}[h!]
\includegraphics[width=\textwidth]{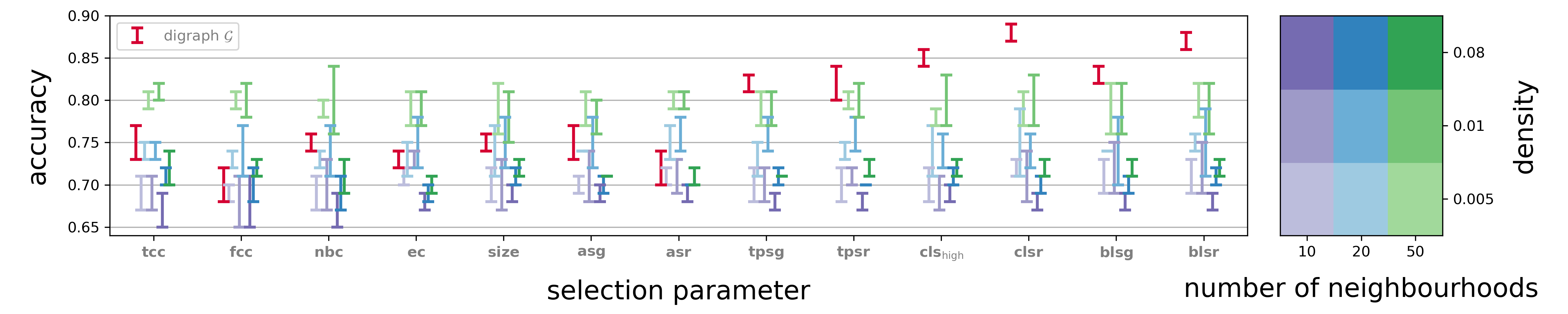}
\caption{Classification of eight random signals on an Erd\H{o}s--R\'enyi random digraph on 1000 vertices and connection probabilities of 8\%, 1\% and 0.5\% and selection of 10, 20, and 50 neighbourhoods, modelled on a NEST simulator. Selection parameters are the same as in the main example and feature parameter is always \textbf{size}. Graph $\GG$ means the BBP graph and its performance with respect to  \textbf{size} as feature parameter is given for comparison. Compare with Supplementary Figure 5.}
\label{fig:NEST}
\end{figure}

\chg{Interestingly, the middle graph density of $0.01$ consistently performed equally as well or better than both the denser $0.08$ and less dense $0.005$ across all feature parameters, except neighbourhood size ({\bf size}) and adjacency spectral gap ({\bf asg}). Another interesting observation is that the strongest selection parameter in this experiment turns out to be normalised Betti coefficient ({\bf nbc}), or transitive clustering coefficient ({\bf tcc}), depending on if ``strongest'' is taken to mean with the highest individual accuracy or with the highest average accuracy from cross-validation, respectively. Both of these selection parameters in the Blue Brain Project experiments exhibited rather mediocre performance (see Figure \ref{Fig:classification}, left).} This suggests that different networks and binary dynamics on them may require experimentation with a collection of selection (and feature) parameters, in order to optimise the classification accuracy.

\section{Discussion}
\label{sec:summary}

In this paper we examined the concept of a closed  neighbourhood in relation to the classification of binary dynamics on a digraph.  \chg{Regardless of what the source of the binary dynamics is, but with the assumption that it is given in a time series of labelled instantiations, we ask  how can the  dynamics be read off and classified. In the context of neuroscience, which is our primary motivation for this study, this is a question on the boundary between computational neuroscience and machine learning. Our methods provide a method of addressing this question. }

We proposed a methodology that will take as input binary dynamics on a digraph and produce a vector summary of the dynamics by means of combinatorial and/or topological \chg{parameters} of a relatively small number of neighbourhoods. Using this methodology we experimented with a dataset implemented on the Blue Brain Project reconstruction of the neocortical column of a rat, and on an artificial neural network with random underlying graph implemented on the NEST  simulator. In both cases the vector summaries were then run through a support vector machine algorithm that was able to achieve a classification accuracy of up to 88\% for the Blue Brain Project data and up to 81\% for the  NEST data. 

We used the same parameters both for selecting neighbourhoods and for the creation of feature vectors. We saw that certain spectral graph parameters used as selection parameters perform significantly better than more classical parameters such as degree and clustering coefficients. We also observed that the parameters that performed best as feature parameters were the simplest ones, namely \hadgesh{size} and \hadgesh{Euler characteristic}. Comparison to randomly selected neighbourhoods showed that the methodology works reasonably well even without selecting the neighbourhoods in an informed way, but that neighbourhoods selected in a way informed by graph parameters gives in general a better performance with a  much smaller error range.

\chg{Our aim was to demonstrate that certain selections of subgraphs, informed by objective structural parameters,   carry enough information to allow classification of noisy signals  in a network of spiking neurons. In this paper the subgraphs selected are closed neighbourhoods, and the selection criteria are our chosen selection parameters. We did not however show, or attempted to demonstrate, that the use of neighbourhoods as a concept, or graph parameters as a selection mechanism are the best methodology. The same techniques could be applied to other subgraph selections  and other vectorisation methods, which can be analysed by our pipeline with relatively small modifications. }

Another aspect of our ideas that was not exploited at all in this project is the use of more than a single graph parameter in the selection procedure. We did show that different parameters are distributed differently in the Blue Brain Project graph, and hence one may hypothesise   that optimising neighbourhood selection by two or more parameters  may give improved classification accuracy.          

\chg{As our aim was not to obtain the best classification, but rather to provide a good methodology for ingesting binary dynamics  on a digraph and  producing machine learning digestible data stream, we did not experiment with other more sophisticated machine learning algorithms. It is conceivable that doing so may produce even better classification accuracy than what is achieved here.}

\chg{Finally, our approach is closely related to graph neural networks where convolution is performed by aggregating information from neighbourhoods, i.e. for every vertex, features are learned from all the adjacent vertices. The pipeline presented in this paper also takes as input sequences of neural firings and sequences of neuron assemblies which turn the firing patterns into feature values. The interaction of our work and the modelling perspectives from graph neural networks and sequence-to-sequence learning might thus pose an interesting future research question.}

\newpage
\section{Methods}

\subsection{Mathematical Concepts and Definitions}
\label{Section:Defs}
We introduce the basic concepts and notation that are used throughout this article. By a \hadgesh{digraph} we   mean a \hadgesh{finite, directed simple graph}, that is, where reciprocal edges between a pair of vertices are allowed, but  multiple edges in the same orientation between a fixed pair of vertices and self-loops  are not allowed. 

\chg{Topology is the study of topological spaces - a vast generalisation of geometric objects. In this paper we only consider spaces that are built out of simplices. Simplices occur in any dimension $n\geq 0$, where a 0-simplex is a point, a 1-simplex is a line segment, a 2-simplex is a triangle, a 3-simplex a tetrahedron and so forth in higher dimensions. Simplices can be glued together to form a topological space. A good survey for this material intended primarily for readers with a neuroscience background can be found in the Materials and Methods section of \cite{Fund}.}

We now describe a general setup that associates a family of  topological objects with a digraph.  A particular case of this setup is the main object of study in this paper. 

\begin{Defi}\label{Def-Top-Op}
A \hadgesh{topological operator on digraphs} is an algorithm that associates with a digraph $\GG$ a topological space $\Gamma(\GG)$, such that if $\HH\subseteq\GG$ is a subgraph then $\Gamma(\HH)\subseteq \Gamma(\GG)$ as a closed subspace. 
\end{Defi}

That is, a topological operator on digraphs is a functor from the category of digraphs and digraph inclusions to the category of topological spaces and inclusions. The flag complex of $\GG$ (ignoring orientation), the directed flag complex \cite{Flagser}, and the flag tournaplex \cite{GLS}  are examples of such operators.  

\begin{Defi}\label{Def-directed-nbd-community}
Let $\GG = (V,E)$ be a digraph, and let $v_0\in V$ be any vertex. 
\begin{itemize}
\item The \hadgesh{neighbours of $v_0$ in $\GG$} are all vertices $v_0\neq v\in V$ that are incident to $v_0$. 
\item The \hadgesh{open neighbourhood of $v_0$} is the subgraph of $\GG$ induced by the neighbours of $v_0$ in $\GG$. The \hadgesh{closed neighbourhood of $v_0$ in $\GG$} is the subgraph induced by the neighbours of $v_0$ and $v_0$ itself. 
\end{itemize}
We denote the open and closed neighbourhoods of $v_0$ in $\GG$ by \hadgesh{$N_\GG^\circ(v_0)$} and  \hadgesh{$N_\GG(v_0)$} respectively. More generally:
\begin{itemize}
\item  Let $S\subseteq V$ be a subset of vertices. Then $N_\GG^\circ(S)$ denotes the union of open  neighbourhoods  of all $v\in S$. Similarly  $N_\GG(S)$ is the union of all closed neighbourhoods of vertices $v\in S$.
\end{itemize}
\end{Defi}

Notice that if $S=\{v_0, v_1\}$, and $v_0$ and $v_1$ are incident in $\GG$, then $N_\GG^\circ(S) = N_\GG(S)$. In this paper we will  mostly consider  closed neighbourhoods. Neighbourhoods are also used in the paper \cite{Reimann-abdn}, which is closely related to  this article.

\begin{Term}\label{Def-S-tribe}
Let $\GG$ be a digraph and let $S$ be a subset of vertices in $\GG$. Unless explicitly stated otherwise, we shall from now on refer to the closed neighbourhood  of $S$ in $\GG$ simply as the \hadgesh{neighbourhood of $S$ in $\GG$}. In the case where $S$ contains a single vertex $v_0$, we will refer to $v_0$ as the \hadgesh{centre} of $N_\GG(v_0)$. 
\end{Term}

\begin{figure}
	\label{tribe}
\includegraphics[scale=1]{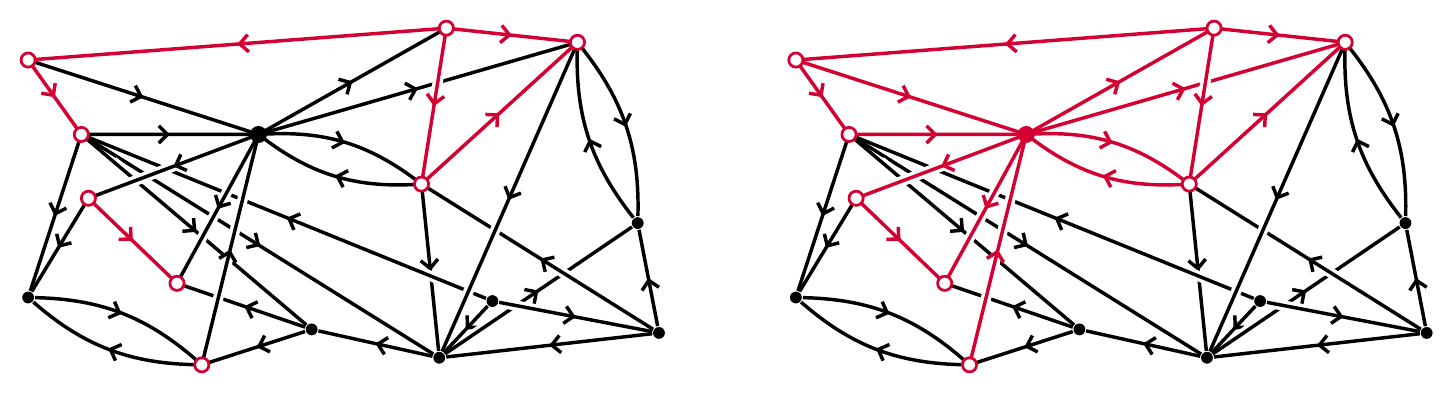}
\caption{An open neighbourhood (left) and a closed neighbourhood (right) in a digraph, marked in red, with its central vertex marked solid colour.}
\end{figure}

The topological operator we consider in this article is the directed flag complex of a digraph which we recall next. See Figure \ref{Fig-clique-complex} for an example.

\begin{Defi}\label{Def-dfl}
A \hadgesh{directed $n$-clique} is a digraph, whose underlying undirected graph is an $n$-clique, and such that the orientation of its edges determines a linear order on its  vertices. An \hadgesh{ordered simplicial complex} is a collection $X$ of finite ordered sets that is closed under subsets. The $n$-simplices of an ordered simplicial complex $X$ are the sets of cardinality $n+1$.  If $\GG$ is a digraph, then the \hadgesh{directed flag complex} associated to $\GG$ is the ordered simplicial complex whose $n$-simplices are the directed $(n+1)$-cliques in $\GG$. We  denote the directed flag complex of a digraph $\GG$ by \hadgesh{$|\GG|$}.
\end{Defi}

\subsection{Encoding binary dynamics on neighbourhoods}
\label{Sec:Tribal-encoding}
We now describe our approach to classification of binary dynamics on a graph in  general terms.  

\begin{Defi}
Let $\GG = (V,E)$ be a graph (directed or undirected). A \hadgesh{binary state} on $\GG$ is  a function $\beta \colon V\to \{0,1\}$. Equivalently, a binary state on $\GG$ is a partition of $V$ into two disjoint subsets that correspond to $\beta^{-1}(0)$ and $\beta^{-1}(1)$, or alternatively as a choice of an element of the power set $\PP(V)$ of $V$. A \hadgesh{binary dynamics} on $\GG$ is a function $B\colon\R_{\ge 0}\to \PP(V)$ that satisfies the following condition:
\begin{itemize}
\item There is a partition of $\R_{\ge 0}$ into finitely many half open intervals $\{[a_i, b_i)\}_{i=1}^P$ for some $P\geq 1$, such that $B$ is constant on $[a_i, b_i)$, for all $i=1,\ldots, P$.
\end{itemize}
\end{Defi}
 Activity in a network of neurons, both natural and artificial, is a canonical example of a binary dynamics on a directed network. 

\subsubsection{Setup.} The task we address in this section is a general classification methodology for binary dynamics functions. Namely, suppose one is  given 
\begin{itemize}
\item a  set of binary dynamics functions $\{B_i\;|\; i\geq 1\}$ on a fixed ambient graph $\GG$,  
\item a set of labels $\LL=\{L_1, L_2,\ldots, L_n\}$, and 
\item a labelling function $L\colon \{B_i\;|\; i\geq 1\}\to \LL$. 
\end{itemize}
In addition, we operate under the assumption that \hadgesh{functions labeled by the same label are variants of the same event} (without specifying what the event is, or in what way its variants are similar). The aim is to produce a topological summary for each $B_i$ in a way that will make the outcome applicable to standard machine learning algorithms. We next describe our proposed mechanism.

\subsubsection{Creation of  vector summary} 
\label{SSec:topological-summary}
Fix a graph $\GG$ and a real-valued  graph \chg{parameter} $Q$, that is, a real-valued function taking digraphs as input and whose values are invariant under graph isomorphisms.  Suppose that a set of labeled binary dynamics functions $\{B^n\}_{n=1}^N$ on $\GG$ is given. Select  an $M$-tuple   $(\HH_1, \HH_2,\ldots, \HH_M)$ of subgraphs of $\GG$, for some fixed positive integer $M$. 

Fix a time interval and divide it into time bins. In each bin, record the vertex set that showed the value 1, that is, was \hadgesh{active} at some point during that time bin. For each $1\le m\le M$, restrict that set to  $\HH_m$ and record the subgraph induced by the active vertices. Apply $Q$ to obtain a numerical $M$-tuple, and  concatenate the vectors into a long vector, which encodes all time bins corresponding to the given dynamics. 

We now describe the procedure more accurately in  three steps.
\begin{enumerate}
	\item[{\crtcrossreflabel{I)}[step1]}] \underline{Interval partition uniformising.} Fix an interval $I=[a,b]\subset\R_{\geq 0}$ and a positive integer $K$. Let $\Delta = \frac{b-a}{K}$. For $1\le k\le K$, let $I_k$ denote the sub-interval
	\[I_k\defeq [a+(k-1)\Delta, a + k\Delta]\subseteq [a,b].\] 
	
	\item[{\crtcrossreflabel{II)}[step2]}] \underline{Subgraph extraction.} For $1\le n\le N$ and each $1\le m\le M$, let $\beta_{m,k}^{n}$ denote the binary state on $\HH_m$ defined by
\[\beta^{n}_{m,k} \defeq \{v\in \HH_m\;|\; \exists t\in I_k, \;\text{such that}\; v\in B^n(t)\}.
\]

	Let $\HH^n_{m,k}\subseteq \HH_m$ be the subgraph induced by all vertices in the set $\beta^{n}_{m,k}$. We refer to $\HH^n_{m,k}$ as the \hadgesh{active subgraph} of $\HH_m$ with respect to the binary dynamics function $B^n$.

	\item[{\crtcrossreflabel{III)}[step3]}] \underline{Numerical featurisation.} For each $1\le n\le N$, let $q^n_{m,k}$ denote the value of  $Q$ applied to  $\HH^n_{m,k}$. Let $F^n$  denote the  $M\times K$ matrix corresponding to the binary dynamics function $B^n$, that is $(F^n)_{m,k} = q^n_{m,k}$.
\end{enumerate}

For use in standard machine learning technology such as support vector machines, we turn the output of the procedure into a single vector by column concatenation. The output of this procedure is what we refer to as a \hadgesh{vector summary of the collection $\{B^n\}_{n=1}^N$} (Figure \ref{Fig:process}). 
It allows great flexibility as its outcome  depends on a number of important choices: 
\begin{itemize}
\item the ambient graph $\GG$,
\item the selection procedure of subgraphs,
\item the interval $I$ and the binning factor $K$, and
\item the graph \chg{parameter} $Q$. 
\end{itemize}
All these choices may be critical to the task of classifying binary dynamics functions, as our use case shows, and have to be determined by experimentation with the data.


\subsection{Selection and feature parameters}
\label{sec:parameters}
In this section we describe the graph parameters used in this article. Some of these parameters are well known in the literature. All of them are invariant under digraph isomorphism. The parameters presented in this section are the primary parameters used for both selection and generation of vector summaries. We chose these particular parameters either because of their prevalence in the literature, or for their strong performance as either selection or feature parameters in classification tasks.  \chg{Other parameters we examined  are mentioned in Supplementary Materials.}

Throughout this section, we let $\GG = (V, E)$ denote a locally  finite digraph (that is, such that every vertex is of finite degree). For $k\geq 1$ and $v_0\in V$, we let $\curs{S}_k(v_0)$ denote the number of directed $(k+1)$-cliques that contain $v_0$. In particular $\curs{S}_1(v_0) = \deg(v_0)$.

\subsubsection{Clustering coefficients.}
\label{sec:cc}
In \cite{WattsStrogatz} Watts and Strogatz introduced an invariant for undirected graphs they called \hadgesh{clustering coefficient}.  For each vertex $v_0$ in the graph $\GG$, one considers the quotient of the number $t_{v_0}$ of triangles in $\GG$ that contain $v_0$ as a vertex by the number ${\deg(v_0)\choose 2}$ of triangles in the complete graph on $v_0$ and its neighbourhood in $\GG$. The clustering coefficient of $\GG$ is then defined as the average across all $v_0\in \GG$  of that number. Clustering coefficients are used in applied graph theory as measures of segregation \cite{Rubinov-Sporns}.

\subsubsection{Clustering coefficient for digraphs.} \label{sec:fcc} The Watts--Strogatz clustering coefficient was generalised by  Fagiolo \cite{fagiolo} to the case of directed graphs. Fagiolo considers for a vertex $v_0$ every possible 3-clique that contains $v_0$, and then identifies pairs of them according to the role played by $v_0$, as a source, a sink, or an intermediate vertex (see Figure \ref{Fig:triangles}, (A), (B) and (C)). Fagiolo also considers cyclical triangles at $v_0$ and identifies the two possible cases of such triangles (see Figure \ref{Fig:triangles}, (D)). The Fagiolo clustering coefficient at $v_0$ is thus the quotient of the number of equivalence classes of directed triangles at $v_0$, denoted by $\vec{t}_{v_0}$, by the number of such classes in the complete graph on $v_0$ and all its neighbours in $\GG$. Thus, if $v_0$ is the $i$-th vertex in $\GG$ with respect to some fixed ordering on the vertices, and $A=(a_{i,j})$ is the adjacency matrix for $\GG$, then 
\[\vec{t}_{v_0} \defeq \frac12\sum_{j,k}(a_{i,j} + a_{j,i})(a_{i,k}+a_{k,i})(a_{j,k}+a_{k,j}),\]
and the clustering coefficient at $v_0$ is defined by 
\[C_F(v_0) \defeq \frac{\vec{t}_{v_0}}{\deg(v_0)(\deg(v_0)-1) - 2\sum_j a_{i,j}a_{j,i}}.\]

\subsubsection{Transitive clustering coefficient}\label{sec:tcc} A directed 3-clique is also known in the literature as a \hadgesh{transitive 3-tournament}. Our variation on the clustering coefficient, the \hadgesh{transitive clustering coefficient} of a vertex $v_0$ in a digraph $\GG$, is the quotient of the number of directed 3-cliques in $\GG$ that contain $v_0$ as a vertex by the number of theoretically possible such 3-cliques. 

Let $\indeg(v_0)$ and $\outdeg(v_0)$ denote the in-degree and out-degree of $v_0$. Let $I_{v_0}$, $O_{v_0}$ and $R_{v_0}$ denote the number of in-neighbours (that are not out-neighbours), out-neighbours (that are not in-neighbours) and reciprocal neighbours of $v_0$, respectively. Notice that
\begin{equation}
\indeg(v_0)  = I_{v_0} + R_{v_0}\quad\text{and}\quad \outdeg(v_0) = O_{v_0} + R_{v_0}.
\end{equation} 

We introduce our variation on Fagiolo's clustering coefficient.

\begin{Defi}
\label{Def:lcc}
	Define the \hadgesh{transitive clustering coefficient at $v_0$} by 
	\begin{equation*}
	C_T(v_0) \defeq \frac{\curs{S}_2(v_0)}{\deg(v_0)(\deg(v_0)-1) - (\indeg(v_0)\outdeg(v_0) + R_{v_0})}.
	\end{equation*}
	\label{Def-local-clustering-coefficient}
\end{Defi}

\begin{figure}
\centering
\includegraphics[width=14cm]{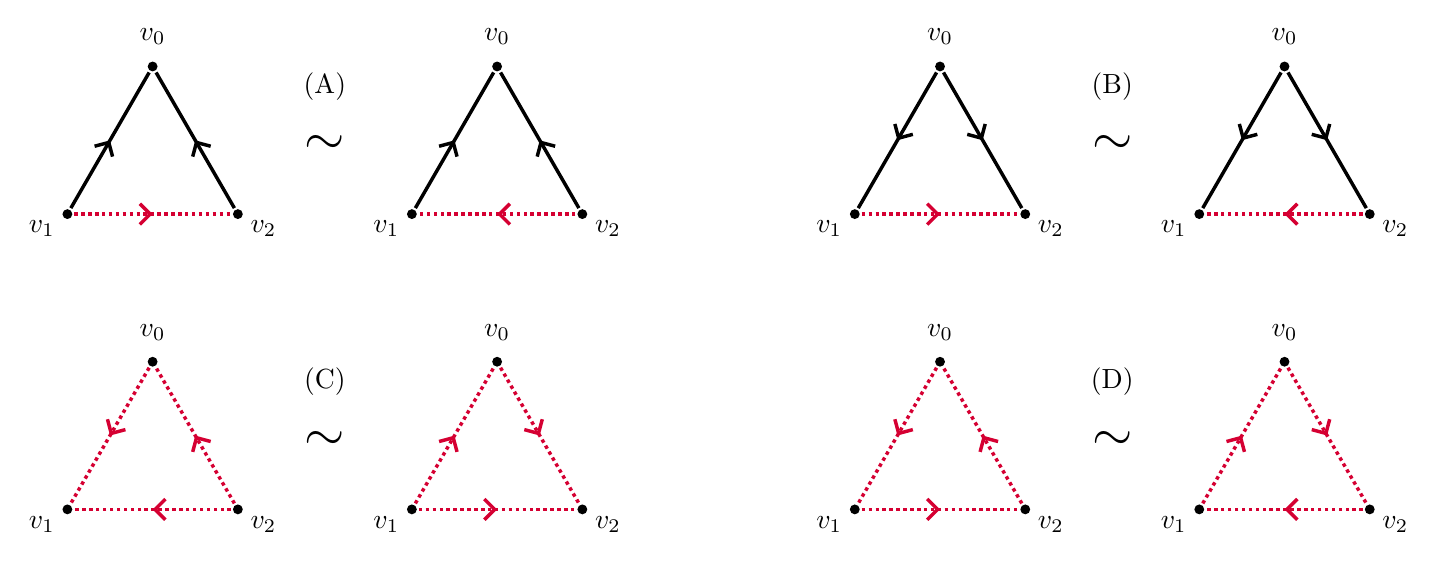}
\caption{Eight possible directed triangles on the same three vertices. The pairs correspond to the identifications made by Fagiolo, with changes denoted by dotted edges. In the definition of the transitive clustering coefficient, the triangles in (A), (B) and (C) are counted individually, and those in (D) are ignored.}
\label{Fig:triangles}
\end{figure}

A justification for the denominator in the definition is needed and is the content of the Lemma 1 
in Supplementary Materials.

Let $A = (a_{i,j})$ denote the adjacency matrix for $\GG$ with respect to some fixed ordering on its vertices. Then for each vertex $v_0\in\GG$ that is the $i$-th vertex in this ordering, $\curs{S}_2(v_0)$ can be computed by the formula
\begin{equation}
\curs{S}_2(v_0) = \sum_{j,k} (a_{i,j}+a_{j,i})(a_{i,k}+a_{k,i})(a_{j,k}+a_{k,j}) - a_{i,j}a_{j,k}a_{k,i} = 2\vec{t}_{v_0} -  \sum_{j,k} a_{i,j}a_{j,k}a_{k,i}.
\end{equation}

\subsubsection{Euler characteristic and normalised Betti coefficient.}
\label{sec:ec-nbc}
The Betti numbers of  the various topological constructions one can associate to a digraph have been shown in many works to give information about structure and function in a graph. A particular example, using Blue Brain Project data is  \cite{Fund}. 

\subsubsection{Euler characteristic.}\label{sec:ec} The Euler characteristic of a complex is possibly the oldest and most useful topological \chg{parameter}, and has been proven to be useful to theory and applications. In the setup of a directed flag complex (or any finite semi-simplicial set) the Euler characteristic is given as the alternating sum of simplex counts across all dimensions: 
\[EC(X) \defeq \sum_{n\geq 0} (-1)^n|X_n|,\]
where $|X_n|$ is the number of $n$-simplices in $X$. Alternatively, the Euler characteristic can be defined using the homology of $X$ by 
\[EC(X) \defeq \sum_{n\geq 0}(-1)^n \dim_\F(H_n(X, \F)),\]
where $\F$ is any field of coefficients. The Euler characteristic is a homotopy invariant, and can take positive or negative values according to the dominance of odd- or even-dimensional cells in the complex in question. 

\subsubsection{Normalised Betti coefficient.}\label{sec:nbc} The \hadgesh{normalised Betti coefficient} is based on a similar idea to the Euler characteristic. It is  invariant under graph isomorphism, but is not a homotopy invariant. Also, unlike the Euler characteristic, it is not independent of the chosen field of coefficients.  We view the normalised Betti coefficient as a measure of how ``efficient'' a digraph is in generating homology, without reference to any particular dimension, but with giving increasing weight to higher dimensional Betti numbers. 

Let $\GG$ be a digraph, and for each  $k\geq 0$, let $s_k(\GG)$ denote the number of $k$-simplices in the directed flag complex $|\GG|$. Fix some field $\F$. By the \hadgesh{Betti number}  $\beta_i$ of $\GG$ we mean the dimension of the homology vector space $H_i(|\GG|, \F)$. 

\begin{Defi}
	\label{Def-Normalised-Betti}
	Let $\GG$ be a locally finite digraph. Define the \hadgesh{normalised Betti coefficient} of $\GG$ to be 
	\begin{equation*}
	\mathfrak{B}(\GG) \defeq \sum_{i=0}^\infty\frac{(i+1)\beta_i(\GG)}{s_i(\GG)}.
	\end{equation*}
\end{Defi}

Normalised Betti coefficients can be defined by any linear combination of Betti numbers, and also in a much more general context (simplicial posets), which we did not explore.  Both the Euler characteristic and the normalised Betti coefficients are invariants of digraphs, and to use them as vertex functions we consider their value on the neighbourhood of a vertex. 

\subsubsection{Size (vertex count).}\label{sec:size}
The \hadgesh{size} of a digraph can be interpreted in a number of ways. One standard way to do so is for a fixed simplicial object associated to a digraph, one counts the number of simplices in each dimension. This will typically produce a vector of positive integers,  the (euclidean) size of which one can consider as the size of the digraph. Alternatively, the simplex count in any dimension can also be considered as a measure of size. In this article we interpret size as the number of vertices in the digraph. Thus by \hadgesh{size} of a vertex $v_0\in \GG$ we mean the vertex count in $N_\GG(v_0)$. When working with binary states on a digraph, neighbourhood size means the number of vertices that obtain the value 1 in $N_\GG(v_0)$.

\subsubsection{Spectral invariants.}
\label{sec:spectral}
The \hadgesh{spectrum} of a (real valued) square matrix or a  linear operator $A$ is the collection of its eigenvalues. \hadgesh{Spectral graph theory} is the study of spectra of matrices associated to graphs. It is a well developed part of combinatorial  graph theory and one that finds many applications in network theory, computer science, chemistry and many other subjects (See a collection of web links on  \href{https://sites.google.com/site/spectralgraphtheory/}{Applications of Spectral Graph Theory}). The various versions of the Laplacian matrix associated to a graph  plays a particularly important role. An interesting work relating neuroscience and the Laplacian spectrum is \cite{Laplacian}.

The \hadgesh{spectral gap} is generally defined as the difference between the two largest moduli of eigenvalues of $A$. In some situations, for instance in the case of the Laplacian matrix, the spectral gap is defined to be the  smallest modulus of nonzero eigenvalues. Given a matrix and its spectrum, either number can be computed. As a standard in this article spectral gaps are considered as the first type described above, except for the  Chung Laplacian spectrum, where the spectral gap is defined to be the value of the minimal nonzero eigenvalue. However, in several cases we considered both options. To emphasise which option is taken we decorated the parameter codes from Table \ref{tab:parameters} with a subscript ``high'' (referring to the difference between the two largest moduli) or ``low'' (referring to the smallest modulus of a nonzero eigenvalue). For example, Figures \ref{fig:plot-val3.2}, \ref{fig:plot-val3.1}, \ref{fig:plot-val4} have $\mathbf{bls_{low}}$ as a parameter, indicating the lowest nonzero value in the Bauer Laplacian spectrum (that is, the minimal nonzero eigenvalue of the Bauer Laplacian matrix).
Another variant of the standard concepts of spectra is what we call the \hadgesh{reversed} spectral gap  (Definitions \ref{Def:sg-adj&tp} and \ref{Defi:bl-sg}).

Yet another  common invariant we considered is the \hadgesh{spectral radius} which is the largest eigenvalue modulus of the matrix in question. We consider here four matrices associated to digraphs: the  adjacency matrix, the transition probability matrix, the Chung Laplacian and the Bauer Laplacian, with details to follow.

\subsubsection{The adjacency and transition probability matrices.} \label{sec:a-tp}   Let $\GG = (V,E)$ be a weighted directed graph with weights $w_{u,v}$ on the edge $(u,v)$ in $\GG$, where $w_{u,v} = 0$ if $(u,v)$ is not an edge in $\GG$.  Let \hadgesh{$W_\GG = (w_{u,v})$} denote the weighted adjacency matrix of $\GG$. Let \(\outdeg(u)\) denote the out-degree of a vertex \(u\). The \hadgesh{transition probability matrix} for $\GG$ is defined, up to an ordering of the vertex set $V$, to be the matrix $P_\GG$, with 
\begin{equation}
\label{eqn:tpsg}
P_\GG \defeq D_{\outg}^{-1}(\GG) \cdot W_\GG,
\end{equation}
where $D^{-1}_{\outg}(\GG)$ is the diagonal matrix with the reciprocal out-degree $1/\outg(u)$ as the $(u,u)$ entry, if $\outg(u)\neq 0$, else the $(u,u)$ entry is  0.

\begin{Defi}\label{Def:sg-adj&tp}
Let $\GG$ be a digraph with adjacency matrix $A_\GG$ and transition probability matrix $P_\GG$. The \hadgesh{adjacency spectral gap} and the \hadgesh{transition probability spectral gap} of $\GG$ are defined in each case to be the  difference between the two largest moduli of its eigenvalues. 

If we replace in the definition of $P_\GG$ the matrix $D_{\outg}(\GG)$ by $D_{\ing}(\GG)$ of in-degrees, we obtain a variant of the transition probability matrix, which we denote by $P_\GG^{\mathrm{rev}}$, and its spectral gap is referred to as the \hadgesh{reversed transition probability spectral gap}.
\end{Defi}

For our specific application we considered the ordinary (as opposed to weighted) adjacency matrix, namely where all weights $w_{u,v}$ are binary. We considered as parameters the spectral radius of the adjacency and transition probability  matrices.

\subsubsection{The Chung Laplacian.}\label{sec:cl} Chung defined the directed Laplacian for a weighted directed graph in \cite{Chung}. The Perron--Frobenius theorem \cite{HJ} states that any real valued irreducible square matrix $M$ with non-negative entries admits a unique eigenvector, all of whose entries are positive. The eigenvalue for this eigenvector is routinely denoted by $\rho$, and it is an upper bound for any other eigenvalue of $M$. 

If $\GG$ is strongly connected (that is, when there is a directed path between any two vertices in $\GG$), then its transition probability matrix is irreducible, and hence satisfies the conditions of the Perron--Frobenius theorem. Thus $P_\GG$ has an eigenvector, all of whose entries are positive. The \hadgesh{Perron vector} is such an eigenvector $\phi$ that is  normalised in the sense that $\sum_{v\in V}\phi(v) = 1$. Let $\Phi$ denote the diagonal matrix with the $v$-th diagonal entry given by $\phi(v)$, and let $P$ denote the transition probability matrix $P_\GG$. 

\begin{Defi}\label{Defi:ch-sg}
Let $\GG$ be a strongly connected digraph. The \hadgesh{Chung Laplacian matrix} for $\GG$ is defined by
\begin{equation}
\LL \defeq I - \frac{\Phi^{\frac{1}{2}}P\Phi^{-\frac{1}{2}}+\Phi^{-\frac{1}{2}}P^*\Phi^{\frac{1}{2}}}{2},
\end{equation}
where \(P^*\) denotes the Hermitian transpose of a matrix \(P\). The \hadgesh{Chung Laplacian spectral gap} \(\lambda\) for a digraph $\GG$ is defined to be the smallest nonzero eigenvalue of the Laplacian matrix.
\end{Defi}

The Chung Laplacian spectral gap \(\lambda\)  of a strongly connected digraph $\GG$ is related to the spectrum of its transition probability matrix $P$ by \cite[Theorem 4.3]{Chung}, which states that the inequalities
\begin{equation}
\min_{i\neq 0}\left\{1 - |\rho_i|\right\} \le \lambda \le \min_{i\neq 0}\left\{1-\mathrm{Re}(\rho_i)\right\}
\end{equation}
hold, where the minima are taken over all eigenvalues of $P$. The theory in \cite{Chung} applies for strongly connected graphs and we therefore defined the Laplacian spectral gap of a neighbourhood to be  that of its largest strongly connected component. 

We use the spectral gap of the Chung Laplacian for the largest strongly connected component of a neighbourhood as a selection parameter. When used as a feature parameter we consider the spectral gap of the largest strongly connected component of the active subgraph of the neighbourhood. We also use the spectral radius of the Chung Laplacian, both as  selection and feature parameter.

\subsubsection{The Bauer Laplacian.}\label{sec:bl} The requirement that $\GG$ is strongly connected is   a nontrivial restriction, but it is required in order to guarantee that the eigenvalues are real. An alternative definition of a Laplacian matrix for directed graphs that does not require strong connectivity was introduced in \cite{Bauer}.  Let \(C(V)\) denote the vector space of complex valued functions on \(V\). The  Bauer  Laplacian for $\GG$ is the transformation \(\Delta_\GG \colon C(V) \rightarrow C(V)\) defined by
\begin{equation}
\label{eqn:blsg}
\Delta_\GG(f)(v) \defeq 
\begin{cases}
	f(v) - \frac{1}{\indeg(v)}\Sigma_v w_{v,u} f(u), & \ \text{if} \ \indeg(v) \ne 0, \\
	0, & \ \text{otherwise}.
\end{cases}
\end{equation}

If \(\indeg(v) \ne 0\) for all \(v \in V\), then \(\Delta_\GG\) corresponds to the matrix \(\Delta_\GG = I - D_{\ing}^{-1}(\GG)\cdot W_\GG\), where \(D^{-1}_{\ing}(\GG)\) is defined analogously to $D^{-1}_{\outg}(\GG)$ in Definition \ref{Def:sg-adj&tp}, and \(W_\GG\) is the weighted adjacency matrix. In our case $W$ is again taken to be the ordinary binary adjacency matrix. 

\begin{Defi}\label{Defi:bl-sg}
The \hadgesh{Bauer Laplacian spectral gap} is  the difference between the two largest moduli of eigenvalues in the spectrum. \end{Defi}
We also considered the spectral radius of the Bauer Laplacian. Both are used as selection as well as feature parameters. 
If we replace in the definition $D_{\ing}(\GG)$ by $D_{\outg}(\GG)$ we obtain a matrix $\Delta_\GG^{\mathrm{rev}}$, whose spectral gap we refer to as the \hadgesh{reversed Bauer Laplacian spectral gap}.

\subsection{acknowledgments}
The authors wish to thank Michael Reimann of the Blue Brain Project for supporting this project and sharing his wisdom and knowledge with us, and Daniela Egas Santander for suggestions to advance our ideas. The authors acknowledge support from EPSRC, grant  EP/P025072/ - ``Topological Analysis of Neural Systems'', and from \'Ecole Polytechnique F\'ed\'erale de Lausanne via a collaboration agreement with the University of Aberdeen. Dejan Govc acknowledges partial support from the Slovenian Research Agency programme P1-0292 and grant N1-0083. 

\bibliographystyle{unsrt}
\bibliography{Tribes}


\newpage
\section{Supplementary Material}

\begin{Lem}\label{Lem-possible-transitive-3}
	Let $\GG$ be a digraph and let $v_0\in\GG$ be a vertex. Then the number of possible directed 3-cliques containing $v_0$ is given by 
	\begin{equation}
	\deg(v_0)(\deg(v_0)-1) - (\indeg(v_0)\outdeg(v_0) + R_{v_0}).
	\end{equation}
\end{Lem}
\begin{proof}
	The set of in-neighbours of $v_0$ give rise to $2{I_{v_0}\choose 2} = I_{v_0}(I_{v_0}-1)$ directed 3-cliques containing $v_0$. Similarly the out-neighbours of $v_0$ give rise to $ O_{v_0}(O_{v_0}-1)$ directed 3-cliques containing $v_0$. A choice of each gives an extra $I_{v_0}O_{v_0}$ directed 3-cliques. Next, each reciprocal neighbour together with either an in-neighbour or an out-neighbour gives rise to three directed 3-cliques at $v_0$. The total number of those is $3R_{v_0}(I_{v_0}+O_{v_0})$. Finally, pairs of reciprocal neighbours give rise to six directed 3-cliques at $v_0$, and the total number of those is $6{R_{v_0}\choose 2}=3R_{v_0}(R_{v_0}-1)$. Let $\curs{P}(v_0)$ denote the total number of transitive 3-tournaments that can be formed by $v_0$ and its neighbours. 
	Summing up we have 
	\begin{align*}
	\curs{P}(v_0) &=  I_{v_0}(I_{v_0}-1) + O_{v_0}(O_{v_0}-1) + I_{v_0}O_{v_0} + 3R_{v_0}(I_{v_0}+O_{v_0})+3R_{v_0}(R_{v_0}-1)\\
	&= (I_{v_0}-O_{v_0})^2 + 3(I_{v_0}O_{v_0}+R_{v_0}I_{v_0}+R_{v_0}O_{v_0} +R_{v_0}^2) -(3R_{v_0}+I_{v_0}+O_{v_0})\\
	&= (\indeg(v_0)-\outdeg(v_0))^2 + 3\indeg(v_0)\outdeg(v_0) -(\indeg(v_0)+\outdeg(v_0)) - R_{v_0}\\
	&= (\indeg(v_0)+\outdeg(v_0))^2 - \indeg(v_0)\outdeg(v_0) - \deg(v_0) -R_{v_0}\\
	&= \deg(v_0)(\deg(v_0)-1) - (\indeg(v_0)\outdeg(v_0)+R_{v_0})
	\end{align*}
	as claimed.
\end{proof}

\subsection{Size, distribution and structure of neighbourhoods in a sample digraph}
\label{sec:botany}

We compare neighbourhoods in a sample digraph sorted by the parameters listed in Table \ref{tab:parameters} in terms of some structural features. The digraph $\GG$ we use is the connectivity graph of the Blue Brain Project reconstruction of the cortical microcircuitry in a young rat brain \cite{Cell_paper}.  The data we used  is available at \cite{Rat}. Our classification experiments are done on the same microcircuit. We also applied the same measurements to other collections of digraphs and obtained  different results. Since our aim is primarily to examine possible relationship between structure and function, we do not report those results here. These extended results are presented at \href{https://homepages.abdn.ac.uk/neurotopology/}{Aberdeen Neurotopology Group} webpage.

We considered the top 50 vertices in the graph sorted by the parameters listed in Table \ref{tab:parameters}. For each parameter we computed the size in terms of number of vertices in each neighbourhood and the pairwise intersections, again in terms of the number of vertices in each intersection. In Table \ref{tab:local-botany-50} we report the minimum, maximum and average of these numbers among the 50 neighbourhoods with highest value for each parameter. We also computed the first six  Betti numbers of each neighbourhood and report the average of these numbers for each parameter. Finally, we considered the union of neighbourhoods in decreasing order, sorted by each parameter,  and computed the number of centres required for their neighbourhoods to cover 90\% of the neurons in entire microcircuit (that is, 28,310 neurons). 

We notice that the top 50 centres with respect to the last six \chg{parameters} listed in Table \ref{tab:local-botany-50} tend to generate neighbourhoods of size close or below the average, with relatively very small intersection. This correlates well with their capacity as selection parameters in our experiments (see Figure \ref{Fig:classification}). However, the two types of clustering coefficients, {\bf fcc} and {\bf tcc}, also generate small top neighbourhoods with small intersection, but are not exceptional as selection parameters.

We also examined the distribution of values for each parameter across the entire graph. The outcome is given in Figure \ref{fig:distribution}, which visually justifies considering neighbourhoods with both highest and lowest parameter values. We did not find a  correlation between the distribution of parameter values and their performance as selection or feature parameters. 

We are therefore led to the conclusion that the performance of graph parameters as selection and/or feature parameters cannot be explained by the structural features we examined. This compares well with the conclusion drawn in \cite{Reimann-abdn}, in which similar experiments using the same dataset but with a different methodology yield results that cannot be explained by structural features such as size and mutual intersection.

\begin{figure}[H]
\centering
\renewcommand\arraystretch{1.3}
\scalebox{.8}{\begin{tabular}{r||c|c|c||c|c|c||c|c|c|c|c|c||c}
\textbf{Parameter} & \multicolumn{3}{c||}{\textbf{size}} & \multicolumn{3}{c||}{\textbf{intersection size}} & \multicolumn{6}{c||}{\textbf{Betti numbers}} & \textbf{90\% cover} \\\hline 
& \textit{min} & \textit{max} & \textit{avg} & \textit{min} & \textit{max} & \textit{avg} & $\beta_0$ & $\beta_1$ & $\beta_2$ & $\beta_3$ & $\beta_4$ & $\beta_5$ & \textit{centre count}\\\hline
$\textbf{fcc}$ & 3 &  181 &  87.9 & 0 &  22 &  0.8 & 1 &  11 &  55 &  6 &  0 &  0 & 1591 \\\hline
$\textbf{tcc}$ & 3 &  170 &  86.2 & 0 &  22 &  0.6 & 1 &  10 &  49 &  5 &  0 &  0 & 1280 \\\hline
$\textbf{ec}$ & 1184 &  1633 &  1456.3 & 30 &  241 &  132.0 & 1 &  288 &  13237 &  2463 &  21 &  0 & 204 \\\hline
$\textbf{nbc}$ & 2 &  1184 &  589.9 & 0 &  132 &  21.6 & 1 &  142 &  3047 &  634 &  11 &  1 & 555 \\\hline
$\textbf{size}$ & 1417 &  1633 &  1509.7 & 44 &  241 &  130.3 & 1 &  287 &  11734 &  2310 &  19 &  0 & 179 \\\hline
$\textbf{asg}$ & 945 &  1604 &  1257.0 & 19 &  226 &  116.3 & 1 &  190 &  10362 &  3108 &  43 &  0 & 270 \\\hline
$\textbf{asr}$ & 1120 &  1622 &  1406.9 & 42 &  241 &  146.9 & 1 &  243 &  12603 &  3127 &  38 &  0 & 249 \\\hline
$\textbf{blsg}$ & 20 &  1344 &  555.2 & 0 &  96 &  12.9 & 1 &  111 &  1444 &  162 &  1 &  0 & 239 \\\hline
$\textbf{blsr}$ & 79 &  974 &  398.3 & 0 &  67 &  7.4 & 1 &  63 &  431 &  56 &  0 &  0 & 318 \\\hline
$\textbf{clsg}$ & 8 &  98 &  40.8 & 0 &  5 &  0.2 & 1 &  0 &  0 &  0 &  0 &  0 & 560 \\\hline
$\textbf{clsr}$ & 69 &  814 &  229.3 & 0 &  35 &  2.9 & 1 &  28 &  81 &  7 &  0 &  0 & 1297 \\\hline
$\textbf{tpsg}$ & 8 &  939 &  368.8 & 0 &  65 &  7.5 & 1 &  62 &  1077 &  131 &  1 &  0 & 445 \\\hline
$\textbf{tpsr}$ & 84 &  1166 &  524.4 & 0 &  98 &  11.3 & 1 &  101 &  1105 &  167 &  1 &  0 & 209 \\\hline\hline
\textit{all vertices} &  2 &  1633 &  492.9 & 0 & 241 & 9.9 & 1 &  94 &  1032 &  146 &  1 &  0 & 212
\end{tabular}}
\vspace{10pt}
\renewcommand\arraystretch{1}
\caption{Size, pairwise intersections, average Betti numbers for the top 50 neighbourhoods of each parameter, and 90\% coverage of the graph by neighbourhoods of highest valued centres, by each parameter. The last row is the same among all vertices, with the last entry on the right giving the average number required for 90\% coverage over 50 random permutations.}
\label{tab:local-botany-50}
\end{figure}

\begin{figure}[H]
\centering
\includegraphics[scale=.5]{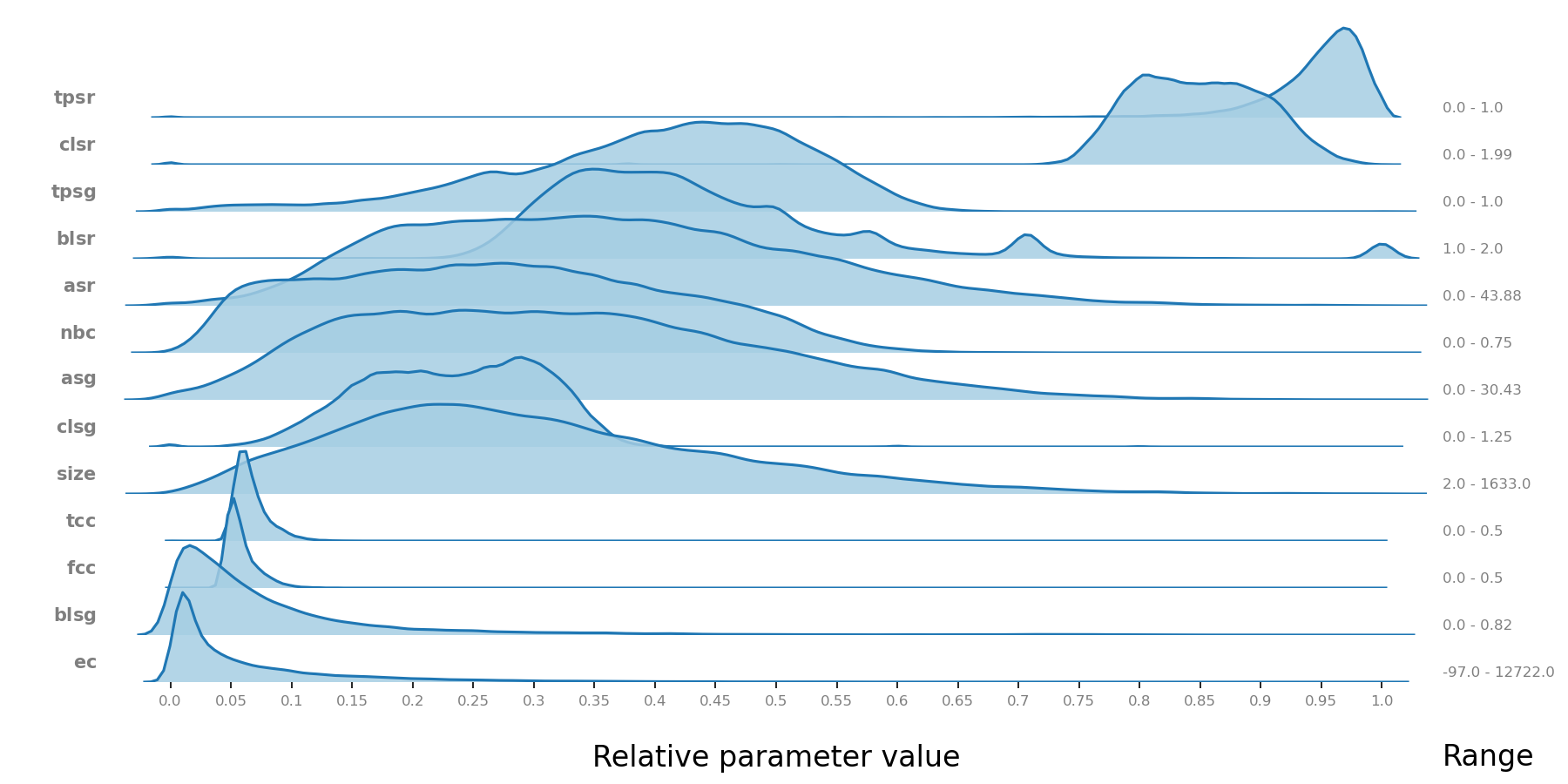}
\caption{Distribution of parameter values across the entire Blue Brain Project microcircuit. The numbers on the right are minimum to maximum values. The values on the $x$-axis are the relative parameter values, rescaled from 0 to 1.  Compare with Figure \ref{fig:distribution-full}}
\label{fig:distribution}
\end{figure}

\chg{The coverage capability of neighbourhoods sorted by various graph and topological parameters is related to another graph theoretic concept. Let $\GG$ be a digraph. }If  $S$ is the entire vertex set of $\GG$, then $N_\GG(S) = \GG$, but  the converse is not true, as $S$ may be much smaller than the full vertex set and still satisfy this condition. Subsets of vertices whose neighbourhoods are the entire  graph are  well studied in graph theory \cite[Section 12.4]{Chartrand-et-al}.

\begin{Defi}\label{Def-domination}
Let $\GG$ be a finite digraph with vertex set $V$. A subset $S\subseteq V$ is  a \hadgesh{dominating set} if $N_\GG(S) =\GG$. The minimum cardinality of a dominating set for $\GG$ is called the \hadgesh{domination number} and is denoted by $\gamma(\GG)$. A dominating set of cardinality $\gamma(\GG)$ is said to be a \hadgesh{minimum dominating set}. 
\end{Defi}

Computing a minimal dominating set is known to be an NP hard problem \cite{Knuth}, though there exist good approximation algorithms. A good summary of the problem and common approaches appears in \cite{Domination}. In Table \ref{tab:local-botany-50} we present, among other computations, the size of  neighbourhoods and the number of neighbourhood from  a sorted list that it takes to cover 90\% of the Blue Brain Project microcircuit. Depending on the selection parameter used, the results are quite different. This suggests that a choice of neighbourhoods informed by certain vertex parameters may give ways of producing more efficient approximation algorithms for the domination number of graphs.

\subsection{Further Graph parameters}
\label{sec:further-parameters}
We describe here further graph and topological parameters we examined.
\subsubsection{Degrees}
For each vertex $v$ in a graph $\GG$, its (total) \hadgesh{degree} $\deg(v)$ is the number of vertices in the open neighbourhood of $v$. The \hadgesh{in- and out-degree} of $v$, denoted $\indeg(v)$ and $\outdeg(v)$ respectively, mean the number of in- and out-neighbours of $v$ respectively. These invariants were examined as graph parameters in our classification algorithm and were found inefficient, except in the case of \hadgesh{size}, which is very closely related to degree and turns out to be the strongest feature parameter we found. 

\subsubsection{Reciprocal degree}
By the  \hadgesh{reciprocal degree} of a vertex $v$ we mean the number of neighbours that are both in-neighbours and out-neighbours. We used reciprocal degree in this work in two ways. The sum of all reciprocal degrees in a neighbourhood (abbreviated {\bf rc}), and the reciprocal degree of the centre ({\bf rc-centre}). 

\subsubsection{Density coefficients}
\label{SSS:dc}
Every $(k+1)$-clique contains $k+1$ $k$-cliques. But no number of $k$-cliques in a graph is guaranteed to form any $(k+1)$-cliques. The \hadgesh{density coefficient} is a ratio of the number of $(k+1)$-cliques by that of $k$-cliques, normalised in its ambient graph. 

\begin{Defi}
\label{Def:dc}
Let $\GG$ be a digraph with $n$ vertices. 
	For $k\geq 2$ define the \hadgesh{$k$-th density coefficient of $\GG$ at $v_0$}  by the formula 
	\begin{equation*}
	D_k(v_0) \defeq \frac{k}{(k+1)(n-k)}\cdot\frac{\curs{S}_k(v_0)}{\curs{S}_{k-1}(v_0)}.
	\end{equation*}
	\label{Def-density-coefficient}
\end{Defi}
The factor $k/(k+1)(n-k)$ normalises the invariant, so that $D_k(v_0)=1$ for every $1 < k < n$ if $v_0$ is a  vertex in $\GG$ that is a complete digraph on $n$ vertices. This is explained in the  next lemma.

\begin{Lem}
\label{Lem:dc}
	For each pair of natural numbers $0<k< n$,  any digraph $\GG$ on $n$ vertices, and any vertex $v_0$ in it,  \[\frac{\curs{S}_k(v_0)}{\curs{S}_{k-1}(v_0)}\le \frac{(k+1)(n-k)}{k}\]
	with equality obtained if and only if $\GG$ is a complete digraph on $n$ vertices.
	\label{Lem-density}
\end{Lem}
\begin{proof}
	We prove the statement by a double counting argument closely following the one given in \cite[Section 10.4]{Jukna}. Let $U$ be the set of all pairs $(\tau,\sigma)$ where $\sigma$ is a directed $(k+1)$-clique containing $v_0$ and $\tau\subseteq\sigma$ is a directed $k$-clique containing $v_0$. Then one can count the number of elements of $U$ in two ways. First, the number of $k$-sub-cliques $\tau$ of a fixed $(k+1)$-clique $\sigma$ containing $v_0$ is exactly $k$, therefore
	\[
	|U|=k \curs{S}_k(v_0).
	\]
	On the other hand, a fixed $k$-clique $\tau$ is a subclique of at most $(n-k)(k+1)$ distinct $(k+1)$-cliques $\sigma$, because there are $(n-k)$ different choices for a vertex that together with $\tau$ will form a $k+1$ clique, and once a vertex was chosen there are $(k+1)$ distinct orientations on the extra $k$ edges, so that the outcome is a directed $(k+1)$-clique. Therefore,
	\[
	|U|\leq (n-k)(k+1)\curs{S}_{k-1}(v_0).
	\]
	Comparing the two expressions, we have:
	\[
	k \curs{S}_k(v_0)\leq (n-k)(k+1)\curs{S}_{k-1}(v_0),
	\]
	which, upon reordering gives the claimed upper bound. Computing the ratio for a complete digraph on $n$ vertices shows that this upper bound is sharp.
\end{proof}
 
 We remark that, while we use the density coefficients as vertex parameters, one can define a global density coefficient on a digraph $\GG$ with vertex set $V$ by 
 \[D_k(\GG) \defeq \frac{1}{|V|}\sum_{v\in V} D_k(v).\]
 By Lemma \ref{Lem:dc}, for any $2\le k\le |V|-1$, $D_k(\GG) = 1$ if and only if $\GG$ is a complete digraph on $V$. 
 Since any digraph on $V$ is a subgraph of the complete digraph on $V$,  $D_k(\GG)$ provides a set of numerical invariants for digraphs, parameterised by dimension (size of clique), which measure a notion of size of the digraph in comparison to the complete digraph on the same vertex set.
 In our specific application, density coefficients did not prove efficient as selection or feature parameters. 

\subsection{Digraph filtrations}
\begin{Defi}\label{Def-tribe-top}
	Let $\GG = (V,E)$ be a digraph, and let $\Gamma$ be a topological operator on digraphs. For a vertex $v\in V$, let \hadgesh{$\Gamma_\GG(v)$} denote $\Gamma(N_\GG(v))$. If $S\subseteq V$ is any subset, let 
	\[\Gamma_\GG(S) \defeq \Gamma(N_\GG(S))=\bigcup_{v\in S} \Gamma_\GG(v).\]
\end{Defi}

Topological operators on digraphs respect inclusions, by definition, and therefore transform a digraph that is filtered by subgraphs into a space that is filtered by closed subspaces.

\begin{Defi}
	\label{Defi-T-filtration}
	Let $\GG = (V,E)$ be a digraph and let $\Gamma$ be a topological operator on digraphs. Fix  a linear ordering $\omega \colon v_1<v_2<\cdots <v_M$ on $V$, where $|V|=M$. For any  integer $n\geq 0$, let $S^\omega_n = \{v\in V\;|\; v\ge v_{M-n}\}$. Define a  filtration  $F^\omega_n(\Gamma(\GG))\subseteq F^\omega_{n+1}(\Gamma(\GG))\subseteq\cdots\subseteq \Gamma(\GG)$   by 
	\[F^\omega_n(\Gamma(\GG)) \defeq \Gamma_{\GG}(S^\omega_n).\]	
	The subspace $F^\omega_n(\Gamma(\GG))$ will be referred to as the \hadgesh{$n$-th $\omega$-filtration layer of $\Gamma(\GG)$}. 
\end{Defi}

From a data analysis point of view filtering $\Gamma(\GG)$, as proposed in Definition \ref{Defi-T-filtration}, can be applied in several ways. In particular, persistent homology \cite{Carlsson} can be used to extract information from the topology in a way that is sensitive to the ordering chosen. \chg{As the orderings can be induced from various vertex functions, the filtrations enable probing into the effect these vertex functions have on the subspace topology. In other words, such filtrations give ways of building $\Gamma(\GG)$ as an increasing union of subspaces, and different choices of orderings may result in totally different sequences of subspaces.  In this article we used graph and topological parameters to determine the ordering on vertices. We also considered only the top (or bottom) of the ordered lists of vertices, and hence studied only the bottom layers of the resulting filtrations. }

\subsection{Data and code} \label{SSec:data-code}
The data used is available at \url{https://doi.org/10.5281/zenodo.4290212}. The entire analysis code can be obtained from \url{https://github.com/JasonPSmith/TriDy}. \chg{The code for the NEST experiments is available at \url{https://github.com/jlazovskis/neurotop-nest/}. The computations for this paper were done using the Maxwell HPC cluster at the University of Aberdeen. To ensure the calculations were computed in a reasonable time frame we used a combination of parallelisation and publicly available packages with efficient algorithms. In particular, the structural parameters of each neighbourhood can be computed independently, so were done simultaneously across multiple nodes and cores. To compute many of the parameters standard python packages were sufficient, such as numpy, scipy and networkx. However, for the more computationally intensive topological parameters we used variations of the Flagser software \cite{Flagser}.}

\begin{landscape}
\begin{figure}[h!]
	\includegraphics[scale=.47]{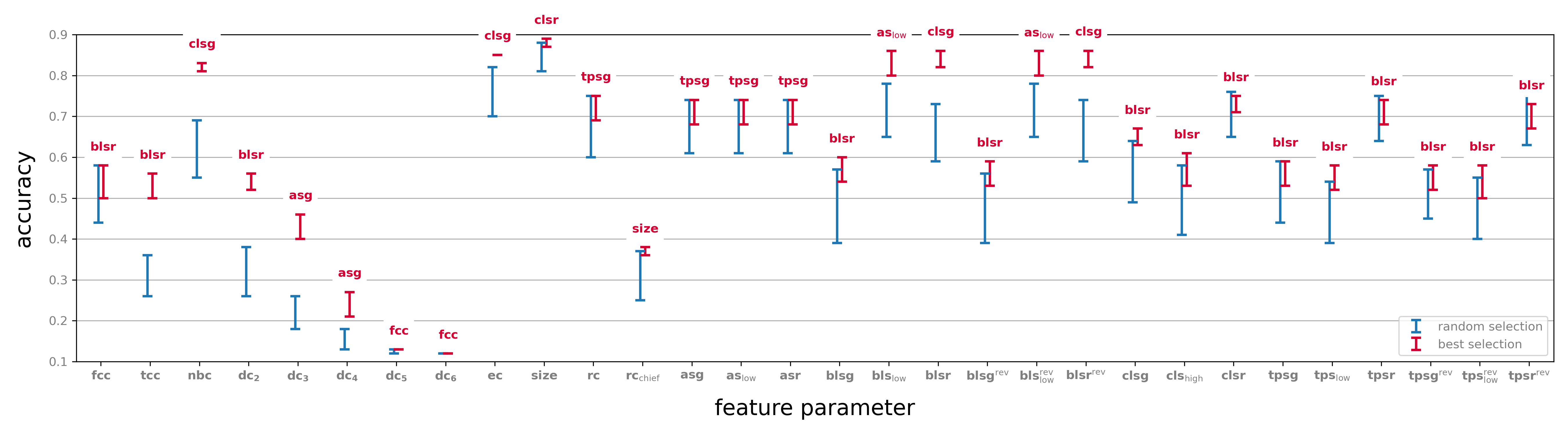}
	\caption{Classification accuracy of signals on the Blue Brain Project microcircuit using 50 randomly selected neighbourhoods, compared with accuracy of neighbourhoods selected by \chg{parameters} with respect to the same feature parameter. Compare with Figure \ref{Fig:comparison-random}.}
	\label{Fig:comparison-random-Supp}
\end{figure}
\end{landscape}

\begin{landscape}
\begin{figure}[ht!]
	\includegraphics[scale=.4]{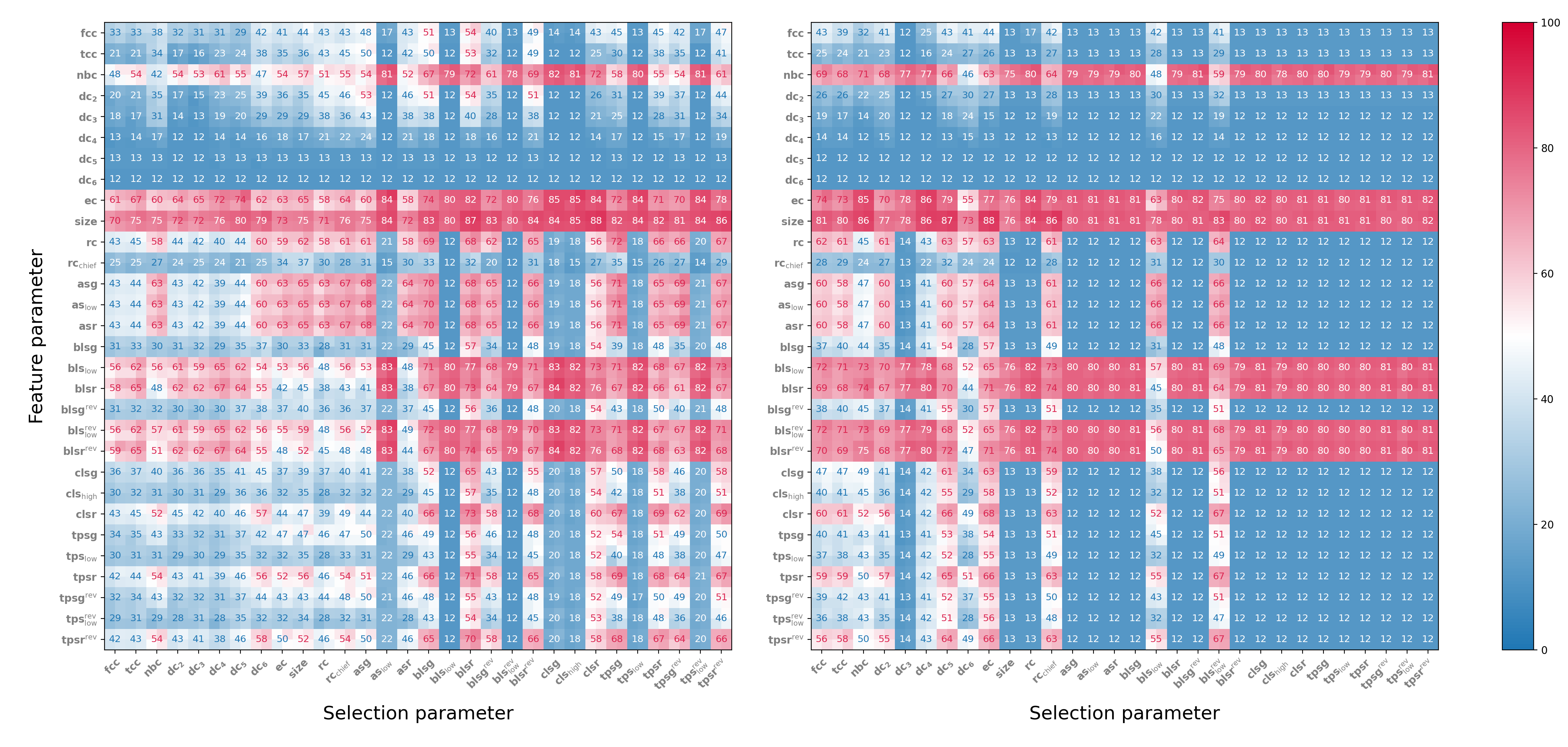}
\caption{Results of classification experiments with respect to all parameters. Left: Classification accuracy with respect to 50 top value vertices by selection parameter. Right: Classification accuracy with respect to 50 bottom value vertices by selection parameter. Compare with Figure \ref{Fig:classification}.}
\label{Fig:classification-all}
\end{figure}
\end{landscape}

\begin{landscape}
\begin{figure}[ht!]
	\includegraphics[scale=.6]{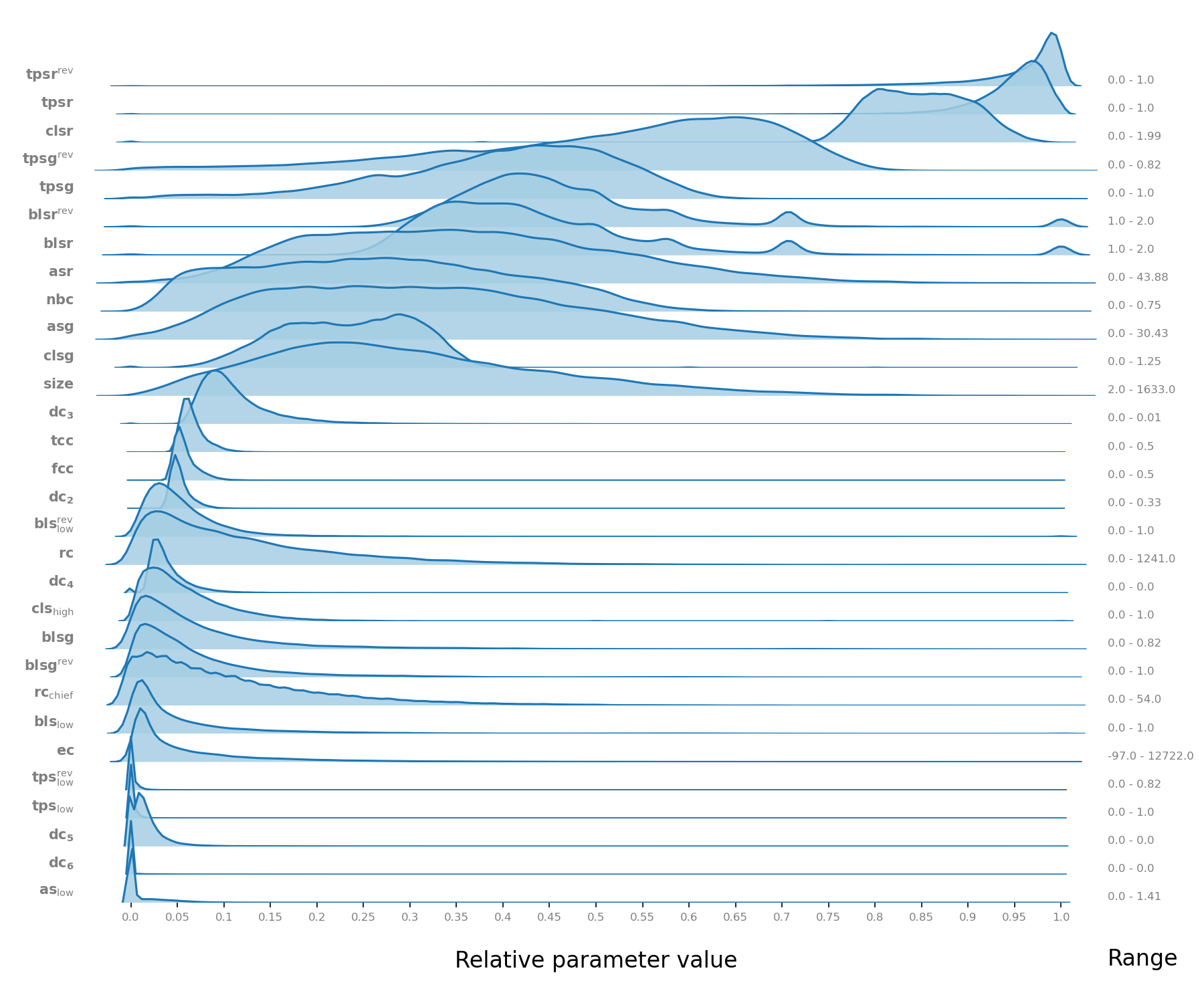}
\caption{Distribution of all parameter values across the entire Blue Brain Project microcircuit. The numbers on the right are minimum to maximum values. The values on the $x$-axis are the relative parameter values, rescaled from 0 to 1. Compare with Figure \ref{fig:distribution}.}
\label{fig:distribution-full}
\end{figure}
\end{landscape}

\begin{landscape}
\begin{figure}[ht!]
	\includegraphics[scale=.47]{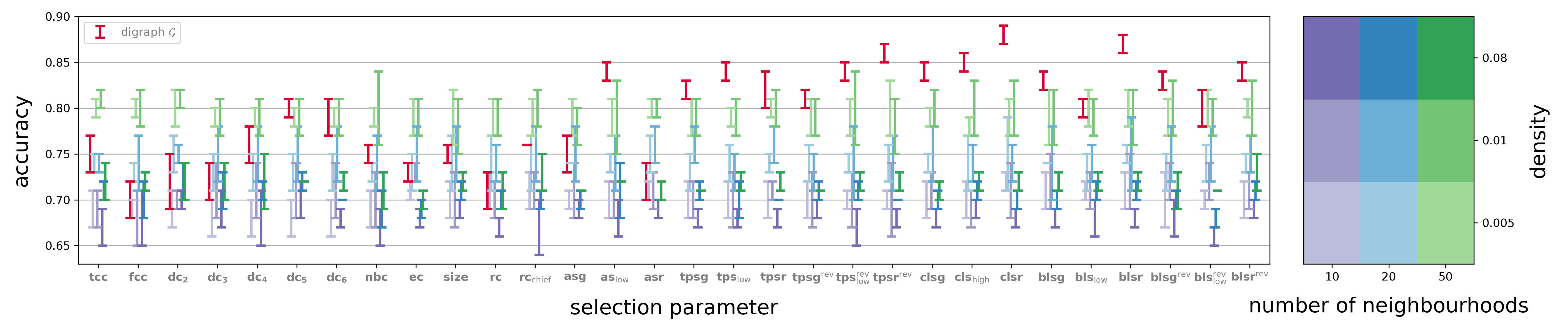}
	\caption{Classification of eight random signals on an Erd\H{o}s--R\'enyi random digraph on 1000 vertices and connection probabilities of 8\%, 1\% and 0.5\% and selection of 10, 20, and 50 neighbourhoods, modelled on a NEST simulator. Selection parameters from Figure \ref{fig:NEST} are included, along with additional parameters. Feature parameter is always \textbf{size}. Graph $\GG$ means the Blue Brain Project graph and its performance with respect to \textbf{size} as feature parameter is given for comparison.}
	\label{Fig:NEST-all}
\end{figure}
\end{landscape}

\end{document}